\documentclass[11pt]{article}

\usepackage[letterpaper,margin=1in]{geometry}
\usepackage{amsmath,amssymb,amsthm}
\usepackage{hyperref} 
\usepackage{enumitem} 
\usepackage{graphicx} 
\usepackage{url}
\usepackage{subcaption}
\usepackage{xcolor}
\usepackage[title]{appendix}

\numberwithin{equation}{section}
\newtheorem{theorem}{Theorem}[section]
\newtheorem{lemma}[theorem]{Lemma}
\newtheorem{definition}[theorem]{Definition}
\newtheorem*{definition*}{Definition}
\newtheorem{cor}[theorem]{Corollary}
\newtheorem{prop}[theorem]{Proposition}

\newtheorem{question}[theorem]{Question}
\newtheorem{conjecture}[theorem]{Conjecture}

\theoremstyle{remark}  

\title{Structural Characterisations of $(n-1,n)$-Trees}
\author{
  Gaurav Kottari\thanks{Corresponding Email: gk917@snu.edu.in},
  Niteesh Sahni\thanks{Email: niteesh.sahni@snu.edu.in}\\
  Department of Mathematics, Shiv Nadar Institution of Eminence,\\ Tehsil Dadri,  Gautam Buddha Nagar, 201314,\\ Uttar Pradesh, India 
}

\date{}

\newcommand{\mscprimary}{05E45}
\newcommand{\mscsecondary}{05C05}

\begin{document}

\maketitle

\begin{abstract}
We study higher-dimensional analogues of graph-theoretic trees within the class of pure $n$-simplicial complexes. Focusing on the case $m=n-1$ in Dewdney’s $(m,n)$-tree framework, we introduce refined notions of path and circuit sequences that overcome the structural limitations of existing definitions. Using these refinements, we establish higher-dimensional analogues of the classical characterisations of trees in graphs, including equivalences based on connectivity, acyclicity, path uniqueness, and enumerative constraints.  We further disprove two conjectures posed by Dewdney by constructing explicit counterexamples, and we formulate corrected versions that hold under additional necessary conditions in the case $m=n-1$. These results provide a structural characterisation of $(n-1,n)$-trees parallel to the classical theory of graph-theoretic trees.
\end{abstract}

\noindent\textbf{Mathematics Subject Classification (2020):} \mscprimary; \mscsecondary

\section{Introduction}\label{Intro}



Several notions of higher-dimensional trees have been developed in the literature, arising from different perspectives. In particular, simplicial trees introduced by Faridi \cite{faridi2002facet, caboara2007simplicial} is rooted in combinatorial commutative algebra, while hypertrees introduced by  \cite{kalai1983enumeration} are defined through acyclic homology. In contrast, Dewdney’s $(m,n)$-trees \cite{dewdney1974higher} provide a purely combinatorial framework based on facet orderings in pure simplicial complexes.

These approaches capture different aspects of higher-dimensional tree-like structures and are not equivalent in general; for instance, there exist simplicial complexes that are $(m,n)$-trees but are not hypertrees in the sense of Kalai. Likewise, there exist simplicial complexes that are simplicial trees in the sense of Faridi but not $(m,n)$-trees in the sense of Dewdney. Concrete examples illustrating these distinctions are provided in the Appendix.

The present work focuses purely on the combinatorial setting and develops a structural theory for the case $(m,n) = (n-1,n)$ within Dewdney’s framework, aiming to recover higher-dimensional analogues of classical graph-theoretic characterisations of trees.

In graph theory, trees admit several equivalent characterisations. In particular, for a finite
graph $G$ with $p$ vertices, the following statements are equivalent \cite{west2001introduction}:
\begin{enumerate}[label=(\roman*)]
    \item \label{1GraphProperty1} $G$ is a tree;
    \item \label{1GraphProperty2} $G$ is connected and has $p-1$ edges;
    \item \label{1GraphProperty3} $G$ is acyclic and has $p-1$ edges;
    \item \label{1GraphProperty4} $G$ is connected and acyclic;
    \item \label{1GraphProperty5} there exists a unique path between any pair of vertices of $G$.
\end{enumerate}

Beineke and Pippert \cite{beineke1969number} introduced the notion of $n$-trees as the first
higher-dimensional generalisation of graph-theoretic trees. In the special case of dimension
two, Beineke and Pippert~\cite{beineke1969characterizations} showed that several of the
characterisations above admit natural analogues.

Subsequently, Dewdney~\cite{dewdney1974higher} introduced the notion of an $(m,n)$-tree, where $n \ge 1$ and $0 \le m \le n-1$, as a generalisation of trees to pure $n$-simplicial complexes. While this framework successfully extends the enumerative characterisations given in \ref{1GraphProperty2}, it does not recover the structural characterisations given in \ref{1GraphProperty3}, \ref{1GraphProperty4} and \ref{1GraphProperty5}.

In this work, we address this structural gap for the case of $(n-1,n)$-trees. We first examine Dewdney’s notions of $(m,n)$-paths, and show that direct analogues of \ref{1GraphProperty5} fail under these definitions by constructing explicit counterexamples. Motivated by this, we introduce restricted $(m,n)$-path sequences, obtained by imposing additional conditions on Dewdney’s paths to eliminate redundancies. Using these restricted paths, we recover an analogue of \ref{1GraphProperty5}.

We then consider the analogue of \ref{1GraphProperty4}. We show that defining acyclicity via the absence of $(m,n)$-circuits is insufficient, and we introduce a refined notion of cyclicity based on a subcomplex of circuits, called Dewdney cycles, which captures the correct higher-dimensional obstruction to $(n-1,n)$-trees.

Although Dewdney~\cite{dewdney1974higher} generalised the classical structural property \ref{1GraphProperty2} to higher dimensions, this higher-dimensional
generalisation leaves certain structural aspects unresolved. To make this precise, we recall his
theorem below.

\begin{theorem}\label{1.1}
A pure $n$-simplicial complex $K$ is an $(m,n)$-tree if and only if the following properties hold for $K$:
\begin{enumerate}
\item $K$ is $(m,n)$-connected.
\item $\alpha_k(K) = \frac{\alpha_0(K)-m-1}{n-m}\binom{n+1}{k+1}-\frac{\alpha_0(K)-n-1}{n-m}\binom{m+1}{k+1}$, for at least one $k$ such that $1 \leq k \leq m$, where $\alpha_k(K)$ is the number of $k$-simplices of $K$.
\end{enumerate}
\end{theorem} 
The characterisation in Theorem~\ref{1.1} depends on the existence of at least one index
$k$ with $1 \le k \le m$ for which the stated relation holds. Consequently, the theorem does not apply when a pure
$n$-simplicial complex $K$ is $(m,n)$-connected and the same relation is satisfied for some indices $k$ with $m+1 \le k \le n$. In the extremal case $m = n-1$, this means that
Theorem~\ref{1.1} does not address the top-dimensional case $k = n$. In the present work, we
show that this limitation can be overcome for $(n-1,n)$-trees.

Additionally, Dewdney~\cite{dewdney1974higher} posed the following two conjectures for general
$(m,n)$-trees.

\begin{conjecture}\label{C1}
A pure $n$-simplicial complex is an $(m,n)$-tree if and only if it has the following two
properties:
\begin{enumerate}
    \item $K$ has no $(m,n)$-circuits.
    \item $\alpha_k(K) =
    \frac{\alpha_0(K)-m-1}{n-m}\binom{n+1}{k+1}
    -\frac{\alpha_0(K)-n-1}{n-m}\binom{m+1}{k+1}$,
    for some $1 \leq k \leq m$.
\end{enumerate}
\end{conjecture}

\begin{conjecture}\label{C2}
A pure $n$-simplicial complex is an $(m,n)$-tree if and only if it has the following two
properties:
\begin{enumerate}
    \item $K$ has no $(m,n)$-circuits.
    \item $\alpha_k(K) =
    \frac{\alpha_0(K)-m-1}{n-m}\binom{n+1}{k+1}
    -\frac{\alpha_0(K)-n-1}{n-m}\binom{m+1}{k+1}$,
    for all $1 \leq k \leq m$.
\end{enumerate}
\end{conjecture}

In this work, we disprove both conjectures by constructing explicit counterexamples. We further
identify additional conditions under which suitably refined versions of these conjectures do
hold in the case of $(n-1,n)$-trees. As a consequence, we obtain a higher-dimensional analogue
of the classical graph-theoretic result that an acyclic graph with $p$ vertices and $p-1$ edges
is a tree, thereby completing the collection of classical structural characterisations in the
setting of $(n-1,n)$-trees.

The structure of the paper is as follows. In Section~\ref{Preliminaries}, we present preliminary definitions and notations that will be used throughout the paper. Section~\ref{Connectivity} revisits Dewdney’s notions of paths, circuits, and connectivity in pure simplicial complexes, introduces reduced simplicial paths and Dewdney cycles, and develops connectivity and acyclicity notions that together recover path uniqueness and enable a tree characterisation analogous to the graph-theoretic setting. In Section~\ref{Simplicial trees}, we study the characterisations of $(n-1,n)$-trees.

\section{Preliminaries}\label{Preliminaries}
In this section, we review the basic notions related to combinatorial simplicial complexes and introduce the terminology required for the study of $(m,n)$-trees. These definitions form the foundation for the refinements and extensions developed in this work.

\subsection{Simplicial complexes}

A simplicial complex $K$ is a collection of non-empty subsets of a finite vertex set $V$, satisfying the condition that every singleton subset is included in $K$, and that every non-empty subset of an element of $K$ is also an element of $K$. The elements of a simplicial complex are referred to as simplices, and the dimension of a simplex is defined as one less than the number of its vertices; in particular, a $k$-simplex is a simplex of dimension $k$ and hence consists of exactly $k+1$ vertices. Accordingly, the dimension of a simplicial complex is the maximum dimension among its simplices. We denote by $ K^k $ the set of all $k$-dimensional simplices in a complex $K$, and by $\alpha_k(K)$ the number of such simplices. Given two simplices $\sigma$ and $\tau$ in a simplicial complex $K$, $\tau$ is a face of $\sigma$ if every vertex of $\tau$ is also a vertex of $\sigma$. A simplicial complex $K$ is called a pure $n$-simplicial complex if every simplex in $K$ is a face of at least one $n$-simplex of $K$. In particular, a pure $1$-simplicial complex $K$ with vertex set $V$ can be viewed as a simple, undirected graph $G(W,E)$ with no isolated vertices, where $W=V$ and $E=K^1$. A subset $L \subseteq K$ of simplices is termed a subcomplex of $K$ if it satisfies the following condition: for every simplex $\sigma$ in $L$, if $\tau$ is a face of $\sigma$ in $K$, then $\tau$ must also be an element of $L$. These definitions are standard and can be found in foundational texts on combinatorial algebraic topology~\cite{kozlov2007combinatorial,zhao2021simplicial}.

\subsection{Dewdney’s framework}\label{dewdneyframework}

Dewdney’s $(m,n)$-tree framework~\cite{dewdney1974higher} provides one of the earliest higher-dimensional analogues of graph-theoretic trees. 
The theory is based on an inductive construction of pure $n$-simplicial complexes via ordered addition of $n$-simplices, together with connectivity notions defined through alternating sequences of simplices. 
Throughout this subsection, we assume $0 \leq m \leq n-1$ unless otherwise stated.

\begin{definition}\label{1 d1}
    Let $S$ be a subset of a simplicial complex $K$. The closure of $S$ in $K$ is the smallest subcomplex $\overline{S}$ of $K$ such that $S \subseteq \overline{S}$, that is,
    $$
    \overline{S} = \{\tau \in K : \tau \text{ is a face of } \sigma \text{ for some } \sigma \in S\}.
    $$
\end{definition}

\begin{definition}\label{1 d2}
    The attachment of an $n$-simplex $\eta_i$ in a pure $n$-simplicial complex $K$ is the subcomplex denoted by $\mathcal{A}(\eta_i,K)$ and is defined as
    \[
\mathcal{A}(\eta_i, K)
=
\left\{
\sigma \subset \eta_i 
\;\middle|\;
\exists\, \eta_j \in K^n \setminus \{\eta_i\}
\text{ such that } \sigma \subset \eta_j
\right\}.
\]

\end{definition}

In other words, $\mathcal{A}(\eta_i,K)$ consists of all faces of $\eta_i$ that are shared with some other $n$-simplex of $K$.

\begin{definition}\label{1 d3}
   A simplicial complex $K$ with $p+1$ vertices is called $m$-complete if the dimension of $K$ is $m$, and every subset of the vertex set of $K$ with $m+1$ elements is a simplex in $K$. In particular, if $m = p$, then $K$ is called a complete simplicial complex.
\end{definition}

\noindent \textbf{Note:} In this work, we adopt a simplified version of the notation and terminology introduced by Dewdney~\cite{dewdney1974higher}. In particular, what we call an $m$-complete ordering corresponds to Dewdney’s $K_{m+1}^m$-ordering, and an $m$-ordering corresponds to his $S_{\geq 1}^m$-ordering. This reformulation reduces notational complexity while preserving the original combinatorial content. With this terminology in place, we now define these notions formally.

\begin{definition}\label{1 d4}
    A pure $n$-simplicial complex $K$ is said to have an $m$-ordering if there exists a sequence of distinct $n$-simplices $\eta_1, \dots, \eta_t$, where $t=\alpha_n(K)$, such that $\mathcal{A}(\eta_i, K_i)$ is an $m$-dimensional simplicial complex for all $2 \leq i \leq t$, where $K_i = \overline{\{\eta_1, \dots, \eta_i\}}$. Furthermore, $K$ is said to have an $m$-complete ordering if $\mathcal{A}(\eta_i, K_i)$ forms a complete simplicial complex with $m + 1$ vertices for all $2 \leq i \leq t$.
\end{definition}

Using the above notions, Dewdney defined $(m,n)$-trees as follows.

\begin{definition}\label{1 d mn tree}
    A pure $n$-simplicial complex $K$ is called an $(m,n)$-tree if $K$ has an $m$-complete ordering.
\end{definition}

To further characterise $(m,n)$-trees, Dewdney~\cite{dewdney1974higher} introduced the notions of $(m,n)$-walks, $(m,n)$-paths, $(m,n)$-connectedness, and $(m,n)$-circuits, which we now recall.

\begin{definition}
    In a pure $n$-dimensional simplicial complex $K$, an $(m,n)$-walk sequence between two $m$-simplices $\sigma_i$ and $\sigma_j$ is an alternating sequence of $m$-simplices and $n$-simplices, given by
    \[
    \sigma_i = \sigma_1, \eta_1, \sigma_2, \dots, \sigma_r, \eta_r, \sigma_{r+1} = \sigma_j,
    \]
    such that for all $1 \leq k \leq r$, the simplices $\sigma_k$ and $\sigma_{k+1}$ are distinct and both are faces of $\eta_k$.  
    The integer $r$ is called the length of the $(m,n)$-walk sequence.
\end{definition}

\begin{definition}\label{1 mnpath}
    An $(m,n)$-walk sequence 
    \[
    \sigma_i = \sigma_1, \eta_1, \sigma_2, \dots, \sigma_r, \eta_r, \sigma_{r+1} = \sigma_j
    \]
    is called an $(m,n)$-path sequence if all simplices in the sequence are distinct.
\end{definition}

\begin{definition}\label{1 d6}
    A pure $n$-simplicial complex $K$ is $(m,n)$-connected if there exists an $(m,n)$-path sequence between every pair of $m$-simplices in $K$.
\end{definition}

\begin{definition}\label{1 d7}
     An $(m,n)$-walk sequence 
     \[
     \sigma_i = \sigma_1, \eta_1, \sigma_2, \dots, \sigma_r, \eta_r, \sigma_{r+1} = \sigma_j
     \]
     is called an $(m,n)$-circuit sequence if it satisfies the following conditions:
\begin{enumerate}
    \item $\sigma_i = \sigma_j$.
    \item For all $p, q$ with $1 \le p, q \le r$ and $p \neq q$, we have $\sigma_p \neq \sigma_q$ and $\eta_p \neq \eta_q$.
\end{enumerate}
\end{definition}

\section{Paths, cycles, and connectivity in pure simplicial complexes}\label{Connectivity}

\subsection{Connectivity and components}
Having introduced $(m,n)$-connectivity in Section~\ref{dewdneyframework}, we now establish a basic result regarding $(n-1,n)$-connectivity.

\begin{prop}
\label{l2-1}
In an $(n-1,n)$-connected pure $n$-simplicial complex $K$, there exists an $(m,n)$-path sequence between any pair of $m$-simplices of $K$ for every $0 \le m \le n-1$.
\end{prop}

\begin{proof}
Let $\tau_i$ and $\tau_j$ be $m$-simplices of a pure $n$-simplicial complex $K$, where $0 \le m \le n-1$. Let $\sigma_{i'}$ and $\sigma_{j'}$ be $(n-1)$-simplices of $K$ such that $\tau_i$ and $\tau_j$ are faces of $\sigma_{i'}$ and $\sigma_{j'}$, respectively. Assume that $\sigma_{i'} \neq \sigma_{j'}$. Since $K$ is $(n-1,n)$-connected, there exists an $(n-1,n)$-path sequence
$$
\sigma_{i'}=\sigma_1,\eta_1,\sigma_2,\ldots,\sigma_r,\eta_r,\sigma_{r+1}=\sigma_{j'}.
$$

Let $p$ be the largest integer $q$, with $1 \le q \le r$, such that $\tau_i$ is a face of $\eta_q$. If $\tau_j$ is a face of $\eta_p$, then
$$
\tau_i,\eta_p,\tau_j
$$
is an $(m,n)$-path sequence.

Suppose that $\tau_j$ is not a face of $\eta_p$. Let $\tau_{p+1}$ be an $m$-face of $\sigma_{p+1}$. If $\tau_j$ is a face of $\eta_{p+1}$, then
$$
\tau_i,\eta_p,\tau_{p+1},\eta_{p+1},\tau_j
$$
is an $(m,n)$-path sequence.

Now suppose that $\tau_j$ is not a face of $\eta_{p+1}$. Since $\sigma_{p+2}$ is distinct from $\sigma_{p+1}$, there exists an $m$-face $\tau_{p+2}$ of $\sigma_{p+2}$ distinct from $\tau_{p+1}$. If $\tau_j$ is a face of $\eta_{p+2}$, then
$$
\tau_i,\eta_p,\tau_{p+1},\eta_{p+1},
\tau_{p+2},\eta_{p+2},\tau_j
$$
is an $(m,n)$-path sequence.

Assume further that $\tau_j$ is not a face of $\eta_{p+2}$. Let $\tau_{p+3}$ be an $m$-face of $\sigma_{p+3}$ distinct from $\tau_{p+2}$.

If $\tau_{p+3}=\tau_{p+1}$ and $\tau_j$ is a face of $\eta_{p+3}$, then
$$
\tau_i,\eta_p,\tau_{p+1},\eta_{p+3},\tau_j
$$
is an $(m,n)$-path sequence.

On the other hand, if $\tau_{p+3}\neq\tau_{p+1}$ and $\tau_j$ is a face of $\eta_{p+3}$, then
$$
\tau_i,\eta_p,\tau_{p+1},\eta_{p+1},
\tau_{p+2},\eta_{p+2},
\tau_{p+3},\eta_{p+3},\tau_j
$$
is an $(m,n)$-path sequence.

If, in both cases, $\tau_j$ is not a face of $\eta_{p+3}$, then we continue this process. At the $k$th step, we choose an $m$-face $\tau_{p+k}$ of $\sigma_{p+k}$ distinct from $\tau_{p+k-1}$ and check whether $\tau_j$ is a face of $\eta_{p+k}$. Whenever a previously chosen $m$-simplex reappears, the resulting sequence may be shortened by omitting the intermediate portion, as in the case $\tau_{p+3}=\tau_{p+1}$ above.

Proceeding in this manner along the finite $(n-1,n)$-path sequence from $\sigma_{i'}$ to $\sigma_{j'}$, we eventually obtain an $(m,n)$-path sequence between $\tau_i$ and $\tau_j$.
\end{proof}

From Proposition~\ref{l2-1}, it follows that if a pure $n$-simplicial complex $K$ is $(n-1,n)$-connected, then for all $0 \leq m \leq n-1$ there exists an $(m,n)$-path sequence between any pair of $m$-simplices of $K$. Consequently, $K$ is $(m,n)$-connected for all $0 \leq m \leq n-1$. We therefore refer to such a complex as \emph{connected}. If \( K \) is not $(n-1,n)$-connected, we refer to it as \emph{disconnected}.

We now examine the basic structural consequences for connected pure $n$-simplicial complexes. In particular, this notion of connectivity induces a decomposition of a pure $n$-simplicial complex into components. This provides a foundation for the structural developments that follow.

We now formalise this by defining components associated with connected pure $n$-simplicial complexes.

\begin{definition}
    A pure $n$-subcomplex $L$ of a pure $n$-simplicial complex $K$ is called a component of $K$ if $L$ is connected and for every $\eta_i \in K^{n}\setminus L^{n}$, the subcomplex $L \cup \overline{\{\eta_i\}}$ is disconnected.
\end{definition}

\begin{figure}[h] 
    \centering
    \includegraphics[width=0.8\textwidth]{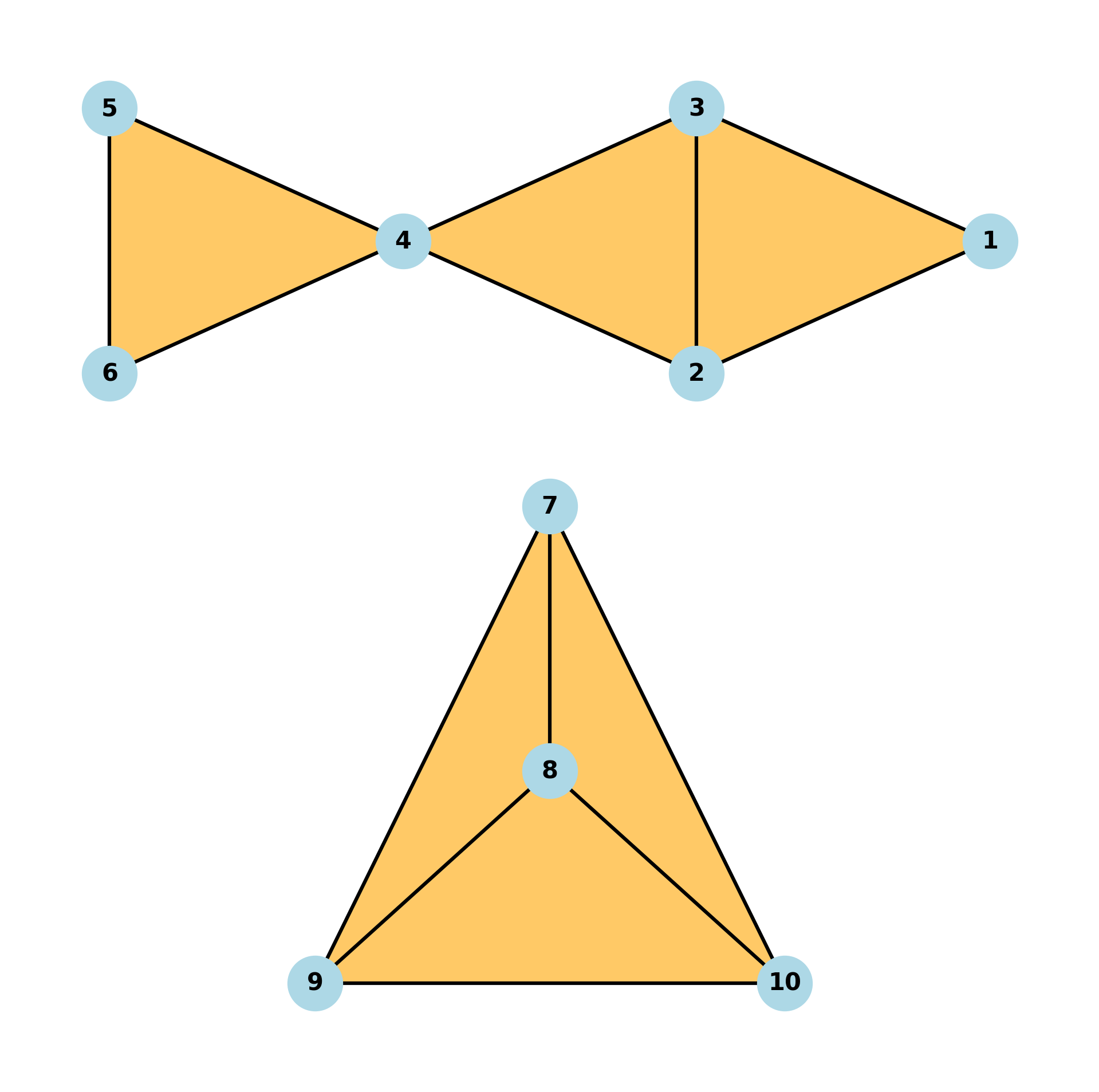} 
    \caption{Example of a pure 2-simplicial complex $K$ and its components}
    \label{fig1} 
\end{figure}
In Figure~\ref{fig1}, the pure 2-simplicial complex 
\[
K = \overline{\{\{1,2,3\}, \{2,3,4\}, \{4,5,6\}, \{7,8,9\}, \{8,9,10\}, \{7,8,10\}\}}
\]
consists of six $2$-simplices, fourteen $1$-simplices, and ten $0$-simplices. Here, the components of $K$ are:
\begin{itemize}
    \item $K_1 = \overline{\{\{4, 5, 6\}\}},$
    \item $K_2 = \overline{\{\{1, 2, 3\}, \{2, 3, 4\}\}},$
    \item $K_3 = \overline{\{\{7, 8, 9\}, \{8, 9, 10\}, \{7, 8, 10\}\}}.$
\end{itemize}

To develop a notion of connectivity analogous to graph components, we now study how $(n-1,n)$-path sequences partition a simplicial complex. 
\begin{prop}\label{3.1}
Let $K$ be a pure $n$-simplicial complex. Then there exists a partition $\mathcal{L}$ of $K^{n}$ such that, for every $L_i \in \mathcal{L}$, the subcomplex $\overline{L_i}$ is a component of $K$.
\end{prop}

\begin{proof}
We show that $(n-1,n)$-path connectivity induces an equivalence relation on the set of $n$-simplices, whose equivalence classes correspond precisely to the components of $K$.

Define a relation $\mathcal{R}$ on $K^n$ by declaring $\eta_i$ and $\eta_j$ to be related if there exist $(n-1)$-faces $\sigma_i$ and $\sigma_j$ of $\eta_i$ and $\eta_j$, respectively, such that an $(n-1,n)$-path sequence exists between them. Since every $n$-simplex has at least two $(n-1)$-faces, $\mathcal{R}$ is reflexive. By definition of an $(n-1,n)$-path sequence, $\mathcal{R}$ is symmetric.

Let $\eta_i, \eta_j, \eta_k \in K^n$ such that $\eta_i \mathcal{R} \eta_j$ and $\eta_j \mathcal{R} \eta_k$. Since $\eta_i \mathcal{R} \eta_j$, there exists an $(n-1,n)$-path sequence
\[
\sigma_i = \sigma_1, \eta_1, \sigma_2, \dots, \sigma_r, \eta_r, \sigma_{r+1} = \sigma_j
\]
between the $(n-1)$-faces $\sigma_i$ and $\sigma_j$. Similarly, since $\eta_j \mathcal{R} \eta_k$, there exists an $(n-1,n)$-path sequence
\[
\tau_j = \tau_1, \eta'_1, \tau_2, \dots, \tau_s, \eta'_s, \tau_{s+1} = \tau_k
\]
between $\tau_j$ and $\tau_k$.

Suppose $\sigma_i = \tau_k$. Since $\eta_i$ and $\eta_k$ are distinct and share a common $(n-1)$-face $\sigma_i$, there exists an $(n-1,n)$-path sequence
\[
\delta_p, \eta_i, \sigma_i, \eta_k, \delta_q,
\]
where $\delta_p$ and $\delta_q$ are $(n-1)$-faces of $\eta_i$ and $\eta_k$, respectively, and both are distinct from $\sigma_i$. Therefore, $\eta_i \mathcal{R} \eta_k$.

Now assume that $\sigma_i$ and $\tau_k$ are distinct.

If $\sigma_j = \tau_j$, then there exists an $(n-1,n)$-walk sequence between $\sigma_i$ and $\tau_k$:
\[
\sigma_i = \sigma_1, \eta_1, \sigma_2, \dots, \sigma_r, \eta_r, \sigma_{r+1}, \eta'_1, \tau_2, \dots, \tau_s, \eta'_s, \tau_{s+1} = \tau_k.
\]

If $\sigma_j \neq \tau_j$, then since both $\sigma_j$ and $\tau_j$ are $(n-1)$-faces of $\eta_j$, there exists an $(n-1,n)$-walk sequence between $\sigma_i$ and $\tau_k$:
\[
\sigma_i = \sigma_1, \eta_1, \sigma_2, \dots, \sigma_r, \eta_r, \sigma_{r+1}, \eta_j, \tau_1, \eta'_1, \tau_2, \dots, \tau_s, \eta'_s, \tau_{s+1} = \tau_k.
\]

In both cases, there exists an $(n-1,n)$-path sequence between $\sigma_i$ and $\tau_k$, which implies that $\eta_i \mathcal{R} \eta_k$. Thus, $\mathcal{R}$ is transitive and hence an equivalence relation.

Consequently, $\mathcal{R}$ partitions $K^{n}$ into equivalence classes. That is,
\[
K^{n} = \bigcup_{l=1}^{t} Cl(\eta_l)
\]
for some $t \in \mathbb{N}$, where
\[
Cl(\eta_x) \cap Cl(\eta_y) = \emptyset
\quad \text{for all } x \neq y.
\]
Therefore,
\[
K = \bigcup_{l=1}^{t} \overline{Cl(\eta_l)}.
\]

We now show that $\overline{Cl(\eta_l)}$ is a component of $K$ for all $1 \le l \le t$.

Let $\tau_p$ and $\tau_q$ be two $(n-1)$-simplices in $\overline{Cl(\eta_l)}$ for some $1 \le l \le t$. Then there exist $n$-simplices $\eta_{p'}$ and $\eta_{q'}$ in $Cl(\eta_l)$ such that $\tau_p$ is a face of $\eta_{p'}$ and $\tau_q$ is a face of $\eta_{q'}$. Assume that $\eta_{p'}$ and $\eta_{q'}$ are distinct. Since $\eta_{p'} \mathcal{R} \eta_{q'}$, there exist $(n-1)$-faces $\sigma_{p'}$ and $\sigma_{q'}$ of $\eta_{p'}$ and $\eta_{q'}$, respectively, such that there exists an $(n-1,n)$-path sequence
\[
\sigma_{p'} = \sigma_1, \eta_1, \sigma_2, \dots, \sigma_r, \eta_r, \sigma_{r+1} = \sigma_{q'}.
\]
Hence one can construct an $(n-1,n)$-walk sequence between $\tau_p$ and $\tau_q$, and therefore an $(n-1,n)$-path sequence between them. Thus, $\overline{Cl(\eta_l)}$ is connected.

Let $\eta_u \in K^n \setminus Cl(\eta_l)$. Then $\eta_u \in Cl(\eta_v)$ for some $v \neq l$. Since $\eta_u$ is not related to any $n$-simplices in $Cl(\eta_l)$, the pure $n$-simplicial complex $\overline{Cl(\eta_l)} \cup \overline{\{\eta_u\}}$ is disconnected. Therefore, $\overline{Cl(\eta_l)}$ is a component of $K$.
\end{proof}

The relationship between connectivity and orderings plays a central role in the theory of pure $n$-simplicial complexes. Dewdney~\cite{dewdney1974higher} stated, without proof, that a pure $n$-simplicial complex is $(m,n)$-connected if and only if it admits an $m$-ordering, and used this equivalence in his development of the theory. To the best of our knowledge, a complete proof of this equivalence in the general case does not appear in the literature. Since this result is used repeatedly in the case $m=n-1$ in the present work, we provide a proof in this special case for completeness.

\begin{theorem}
\label{3.2}
A pure $n$-simplicial complex $K$ is connected if and only if $K$ has an $(n-1)$-ordering.
\end{theorem}

\begin{proof}
Let $K$ be a connected pure $n$-simplicial complex. If $\sigma_i$ and $\sigma_j$ are two $(n-1)$-simplices of $K$, then there exists an $(n-1,n)$-path sequence
\[
\sigma_i = \sigma_1, \eta_1, \sigma_2, \dots, \sigma_r, \eta_r, \sigma_{r+1} = \sigma_j.
\]

If $K^{n} = \{\eta_1, \dots, \eta_r\}$, then by taking $K_i = \overline{\{\eta_1, \dots, \eta_i\}}$ for all $2 \le i \le r$, we conclude that $K$ has an $(n-1)$-ordering.

Let $\eta_k \in K^n \setminus \{\eta_1, \dots, \eta_r\}$, and let $\sigma_k$ be the $(n-1)$-face of $\eta_k$ distinct from $\sigma_j$. Then there exists an $(n-1,n)$-path sequence
\[
\sigma_j = \tau_1, \eta'_1, \tau_2, \dots, \tau_s, \eta'_s, \tau_{s+1} = \sigma_k.
\]
Suppose that
\[
K^n = \{\eta_1, \dots, \eta_r\} \cup \{\eta'_1, \dots, \eta'_s\} \cup \{\eta_k\}.
\]
After concatenating the elements
\[
\eta_1, \dots, \eta_r, \eta'_1, \dots, \eta'_s, \eta_k
\]
in this order and removing duplicates while preserving their first occurrence, we obtain
\[
K^n = \{\eta_1, \dots, \eta_v\},
\]
where $r+1 \le v \le r+s+1$, and $\eta_1, \dots, \eta_v$ are the distinct elements in their order of first appearance. By taking $K_i = \overline{\{\eta_1, \dots, \eta_i\}}$ for all $2 \le i \le v$, we conclude that $K$ has an $(n-1)$-ordering.

If there exists an $n$-simplex of $K$ that is not an element of $\{\eta_1, \dots, \eta_v\}$, then by continuing this process we obtain an $(n-1)$-ordering of $K$.

We prove the converse by contrapositive method. Let $K$ be a pure $n$-simplicial complex that is disconnected, and suppose that $\alpha_n(K) = t$.  Since $K$ is disconnected, it has at least two components.

Let $\eta_1, \dots, \eta_t$ be a sequence of distinct $n$-simplices of $K$. Let $k$ be the least integer $j$ such that $2 \le j \le t$ and $\eta_j$ lies in a component distinct from that of $\eta_1$. Then $\mathcal{A}(\eta_k, K_k)$ is not an $(n-1)$-dimensional simplicial complex, where $K_k = \overline{\{\eta_1, \dots, \eta_k\}}$. Therefore, $K$ has no $(n-1)$-ordering.
\end{proof}

\subsection{Reduced simplicial paths}

We now examine the limitations of Dewdney’s notion of an $(m,n)$-path sequence. 
Unlike the graph-theoretic setting, such path sequences need not be unique in $(n-1,n)$-trees. This failure of uniqueness indicates that Dewdney’s definition is 
too permissive to support higher-dimensional analogues of classical path-based 
characterisations of trees.

The source of this non-uniqueness arises from a local configuration of simplices. Let $K$ be an $(n-1,n)$-tree with at least two $n$-simplices and $n \ge 2$. 
Then there exists an $(n-1)$-simplex $\sigma_k$ that is a face of two distinct 
$n$-simplices $\eta_i$ and $\eta_j$. For any $0 \le m \le n-2$, any two distinct 
$m$-faces of $\sigma_k$ admit two distinct $(m,n)$-simplicial path sequences of length one between them, passing through $\eta_i$ and $\eta_j$, respectively. 
Consequently, path non-uniqueness arises whenever the endpoints lie within such a simplex. 
To isolate this obstruction, we introduce the following notion, which, to the best of our knowledge, is not standard in the literature.

\begin{definition}\label{1joint}
An $m$-simplex of a pure $n$-simplicial complex $K$ is called a 
joint $m$-simplex if it is a face of at least two distinct $n$-simplices.
\end{definition}

Joint simplices therefore account for one fundamental source of path non-uniqueness. Specifically, whenever two $m$-simplices lie in a common joint $(n-1)$-simplex, multiple $(m,n)$-path sequences necessarily exist between them. Excluding such pairs is thus a necessary step toward recovering path uniqueness. However, as we show below, this exclusion alone is not sufficient. Even when the endpoints are not faces of a common joint $(n-1)$-simplex, non-uniqueness may still arise due to ambiguity in how paths traverse successive $n$-simplices. 

This phenomenon is illustrated by a simple example. The simplicial complex
\[
S_1=\overline{\{\{1,2,3\},\{2,3,4\},\{3,4,5\}\}}
\]
is an $(1,2)$-tree, since the sequence of $2$-simplices $\{1,2,3\}, \{2,3,4\}, \{3,4,5\}$ form a $1$-complete ordering of $S_1$. Here, there exist two distinct $(0,2)$-path sequences between the $0$-simplices $\{1\}$ and $\{5\}$, namely,
\[
\{1\},\{1,2,3\},\{2\},\{2,3,4\},\{4\},\{3,4,5\},\{5\}
\quad\text{and}\quad
\{1\},\{1,2,3\},\{3\},\{3,4,5\},\{5\}.
\]

This shows that Dewdney’s definition is too permissive to ensure path uniqueness in $(n-1,n)$-trees. The ambiguity in this example arises from the freedom to pass between consecutive $n$-simplices that need not share an $(n-1)$-face. To restrict this behaviour, one could strengthen the definition of a path by demanding that each pair of consecutive $n$-simplices meet along a common $(n-1)$-face. However, this condition alone is not sufficient to guarantee uniqueness.

Indeed, consider the $(1,2)$-tree
\[
S_2=\overline{\{\{1,2,3\},\{2,3,4\},\{2,3,5\}\}}.
\]
 Here, there exist two distinct $(0,2)$-path sequences between the $0$-simplices $\{1\}$ and $\{4\}$, namely
\[
\{1\},\{1,2,3\},\{2\},\{2,3,5\},\{3\},\{2,3,4\},\{4\}
\quad\text{and}\quad
\{1\},\{1,2,3\},\{2\},\{2,3,4\},\{4\}.
\]
Although consecutive $2$-simplices in both sequences intersect along $1$-faces, the $1$-simplex $\{2,3\}$ appears more than once in the longer sequence. This shows that requiring adjacency via $(n-1)$-faces alone does not prevent non-uniqueness, and that it is necessary to additionally require that all such $(n-1)$-simplices occurring along the sequence be distinct.

However, even with this restriction, non-uniqueness may persist. 
Indeed, in the $(2,3)$-tree
\[
S_3=\overline{\{\{1,2,3,4\},\{2,3,4,5\},\{2,4,5,6\},\{4,5,6,7\}\}},
\]
there exist two distinct $(0,3)$-path sequences between the $0$-simplices $\{3\}$ and $\{6\}$, namely
\[
\begin{aligned}
\{3\},\{2,3,4,5\},\{4\},\{2,4,5,6\},\{6\}
\end{aligned}
\]
and
\[
\begin{aligned}
\{3\},\{1,2,3,4\},\{2\},\{2,3,4,5\},\{4\},\{2,4,5,6\},\{5\},\{4,5,6,7\},\{6\}.
\end{aligned}
\]
 This shows that controlling the transitions between consecutive $n$-simplices alone does not suffice to guarantee uniqueness. In particular, ambiguity may also arise at the endpoints of a path when the initial and terminal $m$-simplices are faces of more than one $n$-simplex along the sequence. To remove this endpoint ambiguity, we impose the additional restriction that both the initial and terminal $m$-simplices be faces of unique $n$-simplices within the sequence.

Motivated by the above examples, we seek a notion of path that eliminates the different sources of ambiguity that arise in $(m,n)$-path sequences. These include ambiguity in the choice of initial and terminal $n$-simplices, repetition of intermediate $(n-1)$-faces, and non-unique transitions between consecutive $n$-simplices. We now formalise this notion.



\begin{definition}\label{1 d}
    An $(m,n)$-path sequence $\sigma_i = \sigma_1,\eta_1,\sigma_2,\dots,\sigma_{r},\eta_{r},\sigma_{r+1}=\sigma_j$ in a pure $n$-simplicial complex $K$ is called a reduced $(m,n)$-path sequence if the following conditions hold:
    \begin{enumerate}
        
        \item \label{cd2} $\sigma_i$ is a face only of $\eta_1$ among the $n$-simplices in the sequence, and $\sigma_j$ is a face only of $\eta_r$.
 
        \item \label{cd3} For all $2 \le z \le r$, the simplex $\sigma_z$ is a face of an $(n-1)$-simplex $\sigma'_z$ of $K$ such that $\sigma'_z$ is a face of both $\eta_{z-1}$ and $\eta_z$. Moreover, for all $x, y$ with $2 \le x, y \le r$ and $x \neq y$, we have $\sigma'_x \neq \sigma'_y$.

    \end{enumerate}
    The subcomplex $P = \overline{\{\eta_1,...,\eta_r\}}$ of $K$ is called a reduced simplicial path.
\end{definition}

We now show that reduced $(n-1,n)$-simplicial paths always exist under $(n-1,n)$-connectivity.

\begin{prop}\label{1propn-1}
    In a connected pure $n$-simplicial complex $K$, there exists a reduced $(n-1,n)$-path sequence between any pair of $(n-1)$-simplices of $K$.
\end{prop}

\begin{proof}
Let $\sigma_i$ and $\sigma_j$ be $(n-1)$-simplices of a pure $n$-simplicial complex $K$. Since $K$ is connected, there exists an $(n-1,n)$-path sequence
\[
\sigma_{i} = \sigma_1, \eta_1, \sigma_2, \dots, \sigma_r, \eta_r, \sigma_{r+1} = \sigma_{j}.
\]

Let $p$ be the largest integer $q$ such that $1 \le q \le r$ and $\sigma_i$ is a face of $\eta_q$. Let $x$ be the smallest integer $y$ such that $p \le y \le r$ and $\sigma_j$ is a face of $\eta_y$. Now, we obtain a reduced $(n-1,n)$-path sequence
\[
\sigma_i, \eta_p, \sigma_{p+1}, \dots, \sigma_x, \eta_x, \sigma_j,
\]
between $\sigma_i$ and $\sigma_j$ in $K$.
\end{proof}

Proposition~\ref{1propn-1} shows that in a connected pure $n$-simplicial complex, a reduced $(n-1,n)$-path sequence exists between every pair of $(n-1)$-simplices. It is natural to ask whether a similar result holds for reduced $(m,n)$-path sequences when $0 \le m \le n-2$. The following example shows that this is not true in general. Even in a connected pure $n$-simplicial complex, there may exist pairs of $m$-simplices for which no reduced $(m,n)$-path sequence exists. Consider the pure $2$-simplicial complex $$K=\overline{\{\{1,2,3\},\{2,3,4\},\{3,4,5\},\{4,5,6\},\{5,6,7\},\{6,7,8\},\{3,7,8\},\{3,8,9\}\}}$$ shown in Figure~\ref{1fig:reducedcounterexample}. In this simplicial complex, no reduced $(0,2)$-path sequence exists between the $0$-simplices $\{1\}$ and $\{9\}$.


\begin{figure}[h]
    \centering
    \includegraphics[width=0.8\textwidth]{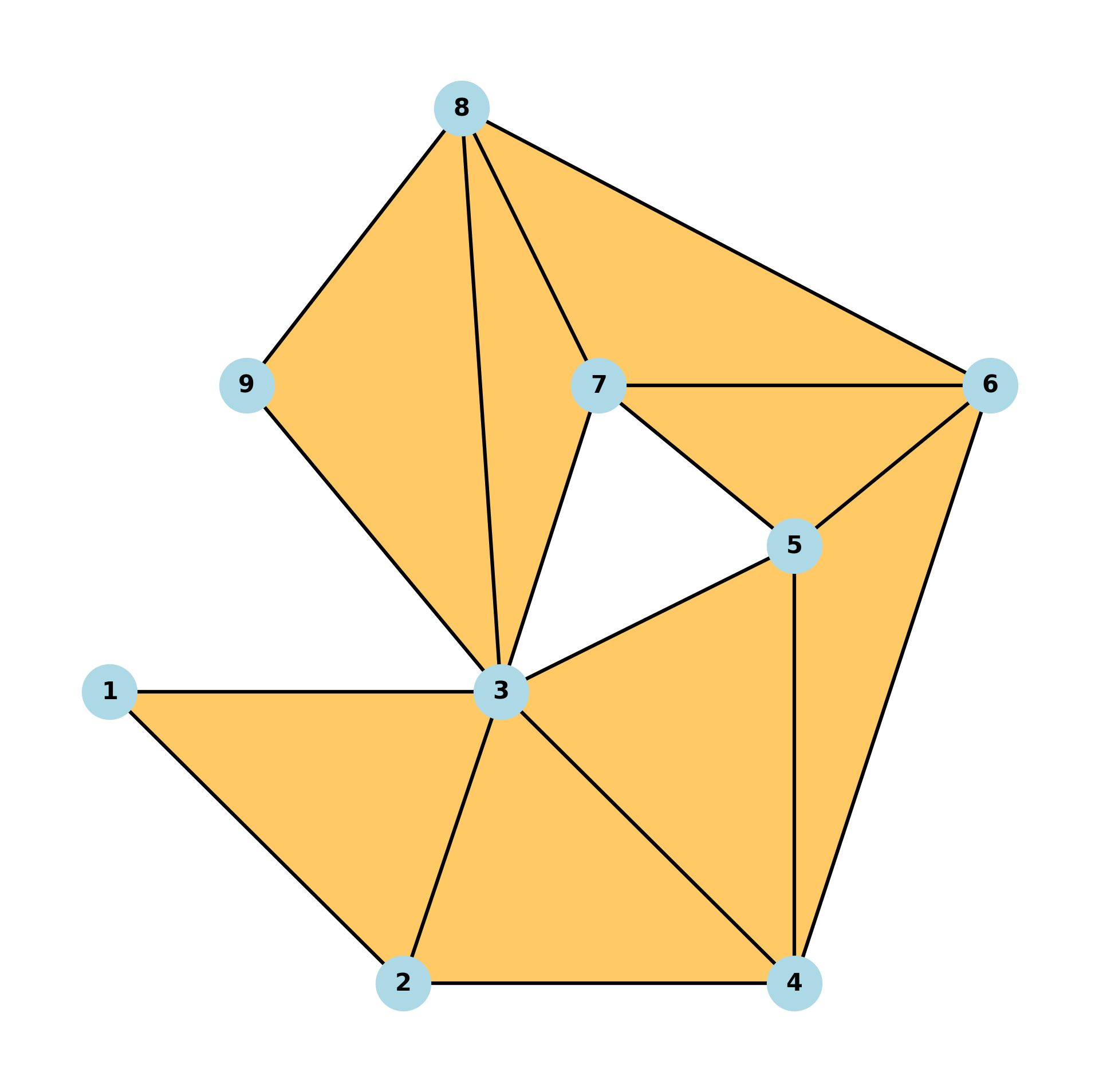} 
    \caption{A connected pure $2$-simplicial complex in which no reduced $(0,2)$-path sequence exists between the $0$-simplices $\{1\}$ and $\{9\}$.}
    \label{1fig:reducedcounterexample} 
\end{figure}

We now establish a uniqueness property of reduced simplicial paths in simplicial complexes without $(n-1,n)$-circuit sequences.


\begin{lemma}
\label{l 3.2}
Let $K$ be a pure $n$-simplicial complex with no $(n-1,n)$-circuit sequences. Suppose there exist two distinct reduced $(m,n)$-path sequences, for some $0 \le m \le n-2$, between $m$-simplices $\sigma_i$ and $\sigma_j$, with associated reduced simplicial paths $P_x$ and $P_y$. If $P_x$ is a subcomplex of $P_y$, then $P_x = P_y$.
\end{lemma}

\begin{proof}
Let $K$ be a pure $n$-simplicial complex and fix an integer $m$ with $0 \le m \le n-2$. Let
\[
\sigma_i = \sigma_1, \eta_1, \sigma_2, \dots, \sigma_r, \eta_r, \sigma_{r+1} = \sigma_j
\]
and
\[
\sigma_i = \tau_1, \eta'_1, \tau_2, \dots, \tau_s, \eta'_s, \tau_{s+1} = \sigma_j
\]
be two reduced $(m,n)$-path sequences between the $m$-simplices $\sigma_i$ and $\sigma_j$ in $K$, such that
\[
\{\eta_1, \dots, \eta_r\} \subseteq \{\eta'_1, \dots, \eta'_s\}.
\]

For all $2 \le z \le r$, there exists an $(n-1)$-simplex $\sigma'_z$ in $K$ satisfying condition~(\ref{cd3}) of the definition of a reduced $(m,n)$-path sequence. Similarly, for all $2 \le k \le s$, there exists an $(n-1)$-simplex $\tau'_k$ in $K$ satisfying condition~(\ref{cd3}). From condition~(\ref{cd2}) of the reduced path sequence definition, we conclude that $\eta_1 = \eta'_1$ and $\eta_r = \eta'_s$.

If $\eta_2 \neq \eta'_2$, then $\eta_2 = \eta'_t$ for some $3 \le t \le s-1$.

\textbf{Case 1:} If $\sigma'_2 \neq \tau'_p$ for all $2 \le p \le t$, then there exists an $(n-1,n)$-circuit sequence
\[
\sigma'_2, \eta'_1, \tau'_2, \eta'_2, \tau'_3 \dots,\tau'_t, \eta'_t,\sigma'_2.
\]

\textbf{Case 2:} If $\sigma'_2 = \tau'_2$, then there exists an $(n-1,n)$-circuit sequence
\[
\tau'_2, \eta'_2, \tau'_3, \dots, \tau'_t, \eta'_t, \tau'_2.
\]

\textbf{Case 3:} If $\sigma'_2 = \tau'_p$ for some $2 < p \le t$, then there exists an $(n-1,n)$-circuit sequence
\[
\tau'_p, \eta'_1, \tau'_2, \dots, \tau'_{p-1}, \eta'_{p-1}, \tau'_p.
\]

In each of the three cases, we obtain a contradiction to the fact that $K$ has no $(n-1,n)$-circuit sequences. Hence, the assumption that $\eta_2 \neq \eta'_2$ is false, and therefore $\eta_2 = \eta'_2$.

Repeating a similar argument, we conclude that $\eta_x = \eta'_x$ for all $3 \le x \le r$.
\end{proof}

\subsection{Dewdney cycles}\label{Simplicial cycles}

In graph theory, a classical result states that a graph is a tree if and only if it is connected and acyclic. In the setting of pure $n$-simplicial complexes, a notion of connectivity based on $(n-1,n)$-paths has already been established. To obtain a higher-dimensional analogue of the graph-theoretic characterisation of trees, it therefore remains to identify an appropriate notion of acyclicity.

A first natural attempt would be to define a pure $n$-simplicial complex to be acyclic if it admits no $(m,n)$-circuit sequences. However, this requirement is too restrictive and excludes some simplicial complexes that are $(n-1,n)$-trees. For example, in the $(1,2)$-tree
\[
T_1=\overline{\{\{1,2,3\},\{1,3,5\},\{1,5,6\},\{3,4,5\}\}},
\]
the sequence
\[
\{1\},\{1,2,3\},\{3\},\{3,4,5\},\{5\},\{1,5,6\},\{1\}
\]
forms a $(0,2)$-circuit sequence. Thus, forbidding all $(m,n)$-circuit sequences rules out complexes that are otherwise tree-like and cannot serve as a suitable notion of acyclicity.

The failure in the previous example arises because successive $n$-simplices in a circuit may be linked through lower-dimensional faces without sufficient restriction. One might therefore attempt to refine the notion of a circuit by requiring that each transition between consecutive $n$-simplices occur through a common $(n-1)$-face. However, this condition alone does not eliminate circuit sequences. Indeed, in the $(1,2)$-tree
\[
T_2=\overline{\{\{1,2,3\},\{2,3,4\},\{2,4,5\}\}},
\]
the sequence
\[
\{2\},\{1,2,3\},\{3\},\{2,3,4\},\{4\},\{2,4,5\},\{2\}
\]
is a $(0,2)$-circuit sequence even though each pair of consecutive $2$-simplices intersects along a common $1$-face. This shows that adjacency via $(n-1)$-faces alone does not prevent cyclic behaviour. In this case, the obstruction appears to arise from endpoint ambiguity, as the initial $0$-simplex reappears as a face of an intermediate $2$-simplex. Thus, additional restrictions are required to control the behaviour of the initial simplex in a circuit sequence.

Even after imposing the above restrictions, short circuit sequences may still occur in an $(n-1,n)$-tree. For example, in the $(1,2)$-tree
\[
T_3=\overline{\{\{1,2,3\},\{2,3,4\}\}},
\]
the sequence
\[
\{2\},\{1,2,3\},\{3\},\{2,3,4\},\{2\}
\]
forms a $(0,2)$-circuit sequence satisfying all previously imposed conditions. To rule out such trivial cases, it is therefore necessary to require that a circuit sequence involve at least three distinct $n$-simplices.

Taken together, the above examples show that a suitable notion of acyclicity must simultaneously control adjacency between $n$-simplices, restrict endpoint behaviour, and exclude short circuit sequences. Motivated by these observations, we now introduce a definition of Dewdney cycles that incorporates all of these requirements.

\begin{definition}\label{1 def c}
    An $(m,n)$-circuit sequence $\sigma_i = \sigma_1, \eta_1, \sigma_2, \dots, \sigma_r, \eta_r, \sigma_{r+1} = \sigma_i$ in a pure $n$-simplicial complex $K$ is called an $(m,n)$-Dewdney cycle sequence if the following conditions hold:
   \begin{enumerate}
       \item $r\geq3$.
       \item For some $2 \leq k \leq r-1$, $\sigma_i$ is not a face of $\eta_k$.
       \item For all $2 \le z \le r$, $\sigma_z$ is a face of an $(n-1)$-simplex 
$\sigma'_z$ of $K$ such that $\sigma'_z$ is a face of both $\eta_{z-1}$ 
and $\eta_z$.

   \end{enumerate}
The subcomplex $C = \overline{\{\eta_1,...,\eta_r\}}$ of $K$ is called Dewdney cycle.

The complex $K$ is called cyclic if it admits an $(m,n)$-Dewdney cycle sequence for some $0 \leq m \leq n-1$, and acyclic otherwise.

\end{definition}

We now give a characterisation of cyclic pure $n$-simplicial complexes in terms of the combinatorial structure of their $n$-simplices.

\begin{theorem}
\label{3.3}
  A pure $n$-simplicial complex $K$ is cyclic if and only if there exists a sequence of at least three distinct $n$-simplices $\eta_1, \dots, \eta_r$ satisfying the following conditions:
\begin{enumerate}
    \item \label{t3,3111} For all $2 \leq z \leq r$, $\eta_{z-1}$ and $\eta_z$ have a common $(n-1)$-face $\sigma_z$.
    \item \label{t3.3 2} For some $0 \leq m \leq n-1$, there exists an $m$-simplex $\tau_1$ that is a face of both $\eta_1$ and $\eta_r$ but not of $\eta_t$ for some $2 \leq t \leq r-1$.
\end{enumerate}

\end{theorem}
    
\begin{proof}
    The direct part is straightforward and follows directly from the definition of a Dewdney cycle sequence.

To prove the converse, consider an alternating sequence of $(n-1)$-simplices and $n$-simplices:
$\sigma_2, \eta_2, \sigma_3, \dots, \sigma_{r-1}, \eta_{r-1}, \sigma_r$. 
Let $p$ be the smallest integer $q$ such that $2 \leq q \leq r-1$ and $\tau_1$ is not a face of $\eta_q$. 
Similarly, let $k$ be the smallest integer $l$ such that $p+1 \leq l \leq r$ and $\tau_1$ is a face of $\eta_l$.

Now consider the alternating sequence of $(n-1)$-simplices and $n$-simplices
\[
\sigma_p, \eta_p, \sigma_{p+1}, \dots, \sigma_{k-1}, \eta_{k-1}, \sigma_k.
\]
Since $\eta_{p-1}$ and $\eta_k$ are distinct, and $\tau_1$ is a face of both $\eta_{p-1}$ and $\eta_k$, we conclude that $\sigma_p \neq \sigma_k$. 
Thus, there exists an $(n-1,n)$-path sequence
\[
\sigma_p = \sigma'_2, \eta'_2, \sigma'_3, \dots, \sigma'_s, \eta'_s, \sigma'_{s+1} = \sigma_k
\]
such that
\[
\{\eta'_2, \dots, \eta'_s\} \subseteq \{\eta_p, \eta_{p+1}, \dots, \eta_{k-1}\}.
\]
Now consider a sequence of distinct $m$-simplices $\tau_2, \dots, \tau_{s+1}$, where $\tau_x$ is a face of $\sigma'_x$ for all $2 \leq x \leq s+1$. Therefore,
\[
\tau_1, \eta_{p-1}, \tau_2, \eta'_2, \tau_3, \dots, \tau_s, \eta'_s, \tau_{s+1}, \eta_k, \tau_1
\]
forms an $(m,n)$-Dewdney cycle sequence in $K$.

\end{proof}

We now have the following corollaries to Theorem \ref{3.3}.

\begin{cor}
\label{c 3.1}
    If a connected pure $n$-simplicial complex $K$ contains a joint $m$-simplex for some $0 \leq m \leq n-2$ such that it is not a face of any joint $(n-1)$-simplex, then $K$ is cyclic.
\end{cor}

\begin{proof}
Let $\tau_k$ be an $m$-face of the $n$-simplices $\eta_i$ and $\eta_j$ of a pure $n$-simplicial complex $K$, for some $0 \leq m \leq n-2$, such that $\tau_k$ is not a face of any joint $(n-1)$-simplex. Let $\sigma_i$ and $\sigma_j$ be $(n-1)$-faces of $\eta_i$ and $\eta_j$, respectively, such that $\tau_k$ is a face of both $\sigma_i$ and $\sigma_j$. Since $K$ is connected, there exists an $(n-1,n)$-path sequence
\[
\sigma_i = \sigma_1, \eta_1, \sigma_2, \dots, \sigma_r, \eta_r, \sigma_{r+1} = \sigma_j.
\]

If $r=1$, then either $\sigma_j$ is an $(n-1)$-face of $\eta_1$ and $\eta_j$, or $\sigma_i$ is an $(n-1)$-face of $\eta_i$ and $\eta_1$. In both cases, this contradicts the choice of the simplex $\tau_k$. If $r=2$, then $\sigma_2$ is an $(n-1)$-face of $\eta_1$ and $\eta_2$ such that $\tau_k$ is a face of $\sigma_2$, again yielding a contradiction.

Therefore, $r \geq 3$. Since $\eta_1$ is distinct from $\eta_2$ and $\tau_k$ is not a face of $\sigma_2$, it follows that $\tau_k$ is not a face of $\eta_2$. Thus, by Theorem~\ref{3.3}, $K$ is cyclic.
\end{proof}

\begin{cor}
  \label{c 3.2}
   If a pure $n$-simplicial complex $K$ contains an $(n-1,n)$-circuit sequence, then $K$ contains an $(n-1,n)$-Dewdney cycle sequence.
  
\end{cor}
Since the proof of Corollary~\ref{c 3.2} is trivial, it is omitted.

\begin{cor}\label{1cormn}
    Let $K$ be a connected, acyclic pure $n$-simplicial complex. Then, for every $0 \le m \le n-2$, there exists a reduced $(m,n)$-path sequence between any pair of $m$-simplices of $K$.
\end{cor}

\begin{proof}
   Let $\tau_i$ and $\tau_j$ be $m$-simplices of a connected, acyclic pure $n$-simplicial complex $K$, where $0 \le m \le n-2$. Let $\sigma_{i'}$ and $\sigma_{j'}$ be $(n-1)$-simplices of $K$ such that $\tau_i$ and $\tau_j$ are faces of $\sigma_{i'}$ and $\sigma_{j'}$, respectively. Assume that $\sigma_{i'}$ and $\sigma_{j'}$ are distinct. Since $K$ is connected, there exists an $(n-1,n)$-path sequence
$$
\sigma_{i'}=\sigma_1,\eta_1,\sigma_2,\dots,\sigma_r,\eta_r,\sigma_{r+1}=\sigma_{j'}.
$$

Let $p$ be the largest integer $q$ such that $1\le q\le r$ and $\tau_i$ is a face of $\eta_q$. Let $x$ be the smallest integer $y$ such that $p\le y\le r$ and $\tau_j$ is a face of $\eta_y$. We show that there exists a sequence of distinct $m$-simplices $\tau_{p+1},\dots,\tau_x$ such that $\tau_k$ is a face of $\sigma_k$ for all $p+1\le k\le x$.

Suppose, to the contrary, that for every sequence of $m$-simplices $\tau_{p+1},\dots,\tau_x$ satisfying $\tau_k\subset \sigma_k$ for all $p+1\le k\le x$, there exist indices $l_1$ and $l_2$ with $p+1\le l_1\neq l_2\le x$ such that $\tau_{l_1}=\tau_{l_2}$.

Since the $n$-simplices $\eta_p,\dots,\eta_x$ are distinct, for each $t$ with $p+1\le t\le x$, there exists a vertex $v_t$ of $\eta_t$ that is not a vertex of $\eta_{t-1}$. Consequently, $v_{t-1}$ is a vertex of $\sigma_t$ for all $p+2\le t\le x$.

Choose an $m$-face $\tau_{p+1}$ of $\sigma_{p+1}$. For each $t$ with $p+2\le t\le x$, choose an $m$-face $\tau_t$ of $\sigma_t$ containing the vertex $v_{t-1}$. By construction, any two consecutive $m$-simplices in the sequence are distinct, and $\tau_t$ is not a face of $\eta_{t-2}$ for
all $p+2 \le t \le x$. However, by assumption, $\tau_{l_1}=\tau_{l_2}$ for some $p+1\le l_1\neq l_2\le x$.

Therefore, there exists a sequence of at least three distinct $n$-simplices
$$
\eta_{l_1},\eta_{l_1+1},\dots,\eta_{l_2-1},
$$
satisfying conditions~(\ref{t3,3111}) and~(\ref{t3.3 2}) of Theorem~\ref{3.3}. This contradicts the acyclicity of $K$. Hence, our assumption is false. Therefore, there exists a sequence of distinct $m$-simplices $\tau_{p+1},\dots,\tau_x$ such that $\tau_k$ is a face of $\sigma_k$ for all $p+1\le k\le x$.

We thus obtain a reduced $(m,n)$-path sequence
$$
\tau_i,\eta_p,\tau_{p+1},\dots,\tau_x,\eta_x,\tau_j,
$$
between $\tau_i$ and $\tau_j$ in $K$.

\end{proof}

From Proposition~\ref{1propn-1} and Corollary~\ref{1cormn}, it follows that if a pure $n$-simplicial complex $K$ is connected and acyclic, then for all $0 \leq m \leq n-1$ there exists a reduced $(m,n)$-path sequence between any pair of $m$-simplices of $K$.

\section{Characterisations of $(n-1,n)$-trees using paths, cycles, and enumeration}\label{Simplicial trees}

We are now in a position to recover higher-dimensional analogues of the classical characterisations of graph-theoretic trees. Using the refined notions of connectivity, reduced paths, and Dewdney cycles developed in Section~\ref{Connectivity}, we now establish equivalences based on connectivity and acyclicity, uniqueness of paths, and enumerative constraints. Together, these results provide a structural characterisation of $(n-1,n)$-trees, fully paralleling the classical theory of trees in graphs.

We begin by generalising the result: A graph is a tree if and only if it is connected and acyclic.

\begin{theorem}
\label{3.5}
    A pure $n$-simplicial complex $K$ is an $(n-1,n)$-tree if and only if it is connected and acyclic.
\end{theorem}

\begin{proof}
Let $K$ be an $(n-1,n)$-tree with $\alpha_n(K)=t$. Then there exists a sequence of distinct $n$-simplices $\eta_1, \dots, \eta_t$ such that, for all $2 \leq i \leq t$, $\mathcal{A}(\eta_i, K_i)$ is a complete simplicial complex with $n$ vertices, where $K_i = \overline{\{\eta_1, \dots, \eta_i\}}$. By Theorem~\ref{3.2}, $K$ is connected.

Suppose there exists an $(m,n)$-Dewdney cycle sequence
\[
\sigma_i = \sigma_1, \eta'_1, \sigma_2, \dots, \sigma_s, \eta'_s, \sigma'_{s+1} = \sigma_i.
\]
Let
\[
j := \max \{\, q \le t \mid \eta_q \in \{\eta'_1,\dots,\eta'_s\} \,\}.
\]
Then $\mathcal{A}(\eta_j, K_j)$ is not a complete simplicial complex with $n$ vertices, leading to a contradiction.
Thus, our assumption that $K$ is cyclic is false, and hence $K$ is acyclic.

Conversely, let $K$ be a connected and acyclic pure $n$-simplicial complex with $\alpha_n(K) = t$. Since $K$ is connected, there exists a sequence $\eta_1, \dots, \eta_l$ of $n$-simplices in $K$, where $2 \leq l \leq t$, such that $\mathcal{A}(\eta_j, K_j)$ forms a complete simplicial complex with $n$ vertices for all $2 \leq j \leq l$. Here, $K_j = \overline{\{\eta_1, \dots, \eta_j\}}$ for all $2 \leq j \leq l$. Let $k$ be the largest integer $l$ with this property.

Assume $k < t$. Then there exists $\eta_p \in K^n \setminus \{\eta_1, \dots, \eta_k\}$. Let $\sigma_k$ be an $(n-1)$-face of $\eta_k$, and let $\sigma_p$ be an $(n-1)$-face of $\eta_p$ such that $\sigma_k \neq \sigma_p$. Consider an $(n-1,n)$-path sequence
\[
\sigma_p = \sigma'_1, \eta'_1, \sigma'_2, \dots, \sigma'_r, \eta'_r, \sigma'_{r+1} = \sigma_k.
\]

We now proceed by considering the following three cases.

\textbf{Case 1:} Suppose $\eta'_1 \in \{\eta_1,\dots,\eta_k\}$.

Let $\eta'_1 = \eta_h$ for some $1 \leq h \leq k$. Then $\eta_p$ and $\eta_h$ share a common $(n-1)$-face $\sigma_p$. By the choice of $k$, there exists a vertex $v_p \in V(\eta_p) \setminus V(\eta_h)$ and some $\eta_y \in \{\eta_1, \dots, \eta_k\} \setminus \{\eta_h\}$ such that $v_p \in V(\eta_y)$.

Since $\overline{\{\eta_1,\dots,\eta_k\}}$ is connected, there exists a sequence of distinct $n$-simplices
\[
\eta_y=\eta''_1,\eta''_2,\dots,\eta''_e=\eta_h,
\]
all belonging to $\{\eta_1,\dots,\eta_k\}$, such that, for all
$1 \leq u \leq e-1$, the simplices $\eta''_u$ and $\eta''_{u+1}$
share a common $(n-1)$-face. Thus, there exists a sequence of at least
three distinct $n$-simplices
\[
\eta''_1, \eta''_2, \dots, \eta''_e, \eta_p,
\]
satisfying conditions~(\ref{t3,3111}) and~(\ref{t3.3 2}) of
Theorem~\ref{3.3}. This contradicts the assumption that $K$ is acyclic.

\textbf{Case 2:} Suppose $\eta'_r \notin \{\eta_1,\dots,\eta_k\}$.

In this case, $\eta'_r$ and $\eta_k$ share a common $(n-1)$-face $\sigma_k$. Using a similar argument as in Case~1, we obtain a contradiction.

\textbf{Case 3:} Suppose $\eta'_1 \notin \{\eta_1,\dots,\eta_k\}$ and $\eta'_r \in \{\eta_1,\dots,\eta_k\}$.

Let $g$ be the largest integer $f$ such that $1 \leq f \leq r - 1$ and $\eta'_f \notin \{\eta_1, \dots, \eta_k\}$. Then $\eta'_g$ and $\eta'_{g+1}$ share a common $(n-1)$-face $\sigma'_{g+1}$, where $\eta'_{g+1} \in \{\eta_1, \dots, \eta_k\}$. Using a similar argument as in Case~1, we again obtain a contradiction.

Thus, our assumption that $k < t$ is false, and hence $k = t$.
\end{proof}

We now generalise the classical characterisation of a tree that a connected graph $G$ is a tree if and only if there exists a unique path between every pair of vertices.

\begin{theorem}\label{thmun}
 A connected pure $n$-simplicial complex is an $(n-1,n)$-tree if and only if it satisfies the following conditions:
\begin{enumerate}
    \item \label{un1}
    There exists a unique $(n-1,n)$-path sequence between any two $(n-1)$-simplices.

    \item \label{un2}
    For all $0 \le m \le n-2$, there exists a unique reduced simplicial path between any two $m$-simplices $\sigma_i$ and $\sigma_j$ that are not faces of a common joint $(n-1)$-simplex.

    \item \label{un3}
    For all $0 \le m \le n-2$, there exists no $(m,n)$-Dewdney cycle sequence.
\end{enumerate}

\end{theorem}

\begin{proof}

Fix an integer $m$ with $0 \le m \le n-2$. Now consider two reduced $(m,n)$-path sequences
\[
\sigma_i = \sigma_1, \eta_1, \sigma_2, \dots, \sigma_r, \eta_r, \sigma_{r+1} = \sigma_j
\]
and
\[
\sigma_i = \tau_1, \eta'_1, \tau_2, \dots, \tau_s, \eta'_s, \tau_{s+1} = \sigma_j
\]
between the $m$-simplices $\sigma_i$ and $\sigma_j$ in an $(n-1,n)$-tree $K$. For all $2 \le z \le r$, there exists an $(n-1)$-simplex $\sigma'_z$ in $K$ satisfying condition~(\ref{cd3}) of the definition of a reduced $(m,n)$-path sequence. Similarly, for all $2 \le k \le s$, there exists an $(n-1)$-simplex $\tau'_k$ in $K$ satisfying condition~(\ref{cd3}).

Let $r = s = 1$. Since $\sigma_i$ and $\sigma_j$ are not faces of a common joint $(n-1)$-simplex, it follows that $\sigma_i \cup \sigma_j$ is not a face of any joint $(n-1)$-simplex. Using the fact that $\sigma_i \cup \sigma_j$ is a face of both $\eta_1$ and $\eta'_1$, and by Corollary~\ref{c 3.1}, we conclude that $\eta_1 = \eta'_1$.

Now let $r \le s$ with $s \ge 2$, and assume that there exists an index $h$ such that
\[
\eta_h \neq \eta'_l \quad \text{for all } 1 \le l \le s.
\]
Let $a$ be the least such integer.

\textbf{Case 1:} $a = 1$.

\textbf{Case 1.1:} Assume that for all $2 \le l_1 \le r$ and $2 \le l_2 \le s$, 
$\eta_{l_1} \neq \eta'_{l_2}$.

If either $\eta_1$ and $\eta'_1$, or $\eta_r$ and $\eta'_s$, share a common $(n-1)$-face, then there exists a sequence of at least three distinct $n$-simplices
\[
\eta_r, \dots, \eta_1, \eta'_1, \dots, \eta'_s
\quad \text{or} \quad
\eta_1, \dots, \eta_r, \eta'_s, \dots, \eta'_1,
\]
satisfying conditions~(\ref{t3,3111}) and~(\ref{t3.3 2}) of Theorem~\ref{3.3}. 
This contradicts the acyclicity of $K$.

Hence assume that neither $\eta_1$ and $\eta'_1$, nor $\eta_r$ and $\eta'_s$, share a common $(n-1)$-face.

Let $\beta_p$ be an $(n-1)$-face of $\eta_r$ containing $\sigma_j$, and let $\beta_q$ be an $(n-1)$-face of $\eta'_s$ containing $\sigma_j$. Then $\beta_p \neq \beta_q$. Consider an $(n-1,n)$-path sequence
\[
\beta_p = \beta_1, \eta''_1, \beta_2, \dots, \beta_t, \eta''_t, \beta_{t+1} = \beta_q.
\]

Assume $t \geq 3$. If $\sigma_j$ is not a face of some $\eta''_l$ with $2 \le l \le t-1$, then, by Theorem~\ref{3.3}, we again obtain a contradiction to the acyclicity of $K$.

Suppose instead that $\sigma_j$ is a face of every $\eta''_l$, $1 \le l \le t$. 
Possibly $\eta''_1 = \eta_r$ or $\eta''_t = \eta'_s$; however, the remaining $t-2$ $n$-simplices are distinct from those in the two reduced path sequences.

Concatenate
\[
\eta_1,\dots,\eta_r,\eta''_1,\dots,\eta''_t,\eta'_s,\dots,\eta'_1
\]
and remove repetitions while preserving first occurrences. This yields a sequence
\[
\eta_1,\dots,\eta_r,\eta''_{a_1},\dots,\eta''_{a_x},\eta'_s,\dots,\eta'_1,
\]
where 
\[
\{\eta''_{a_1},\dots,\eta''_{a_x}\} \subseteq \{\eta''_1,\dots,\eta''_t\},
\]
consisting of at least three distinct $n$-simplices satisfying conditions~(\ref{t3,3111}) and~(\ref{t3.3 2}). 
This again contradicts the acyclicity of $K$.

A similar argument applies when $t=1$ or $t=2$, yielding the same contradiction.

\textbf{Case 1.2:} Suppose that for some $2 \le l_1 \le r$ and $2 \le l_2 \le s$, $\eta_{l_1} = \eta'_{l_2}$.

Let $b$ be the smallest integer $l_1$ such that $2 \le l_1 \le r$ and $\eta_{l_1} = \eta'_{l_2}$ for some $2 \le l_2 \le s$. Then there exists a sequence of at least three distinct $n$-simplices,
\[
\eta_1, \dots, \eta_b, \eta'_{l_2-1}, \dots, \eta'_1,
\]
satisfying conditions~(\ref{t3,3111}) and~(\ref{t3.3 2}) of Theorem~\ref{3.3}. Thus, we obtain a contradiction to the acyclicity of $K$.

\textbf{Case 2:} Suppose $a \ge 2$. Then $\eta_{a-1} = \eta'_c$ for some $1 \le c \le s-1$.

\textbf{Case 2.1:} Suppose that for any $a+1 \le l_1 \le r$ and $c+1 \le l_2 \le s$, $\eta_{l_1} \neq \eta'_{l_2}$.

In this case, there exists a sequence of at least three distinct $n$-simplices
\[
\eta_r, \dots, \eta_{a-1}, \eta'_{c+1}, \dots, \eta'_s
\]
satisfying conditions~(\ref{t3,3111}) and~(\ref{t3.3 2}) of Theorem~\ref{3.3}. Thus, we obtain a contradiction to the acyclicity of $K$.

\textbf{Case 2.2:} Suppose that for some $a+1 \le l_1 \le r$ and $c+1 \le l_2 \le s$, $\eta_{l_1} = \eta'_{l_2}$.

Let $e$ be the smallest integer $l_1$ such that $a+1 \le l_1 \le r$ and $\eta_{l_1} = \eta'_{l_2}$ for some $c+1 \le l_2 \le s$.

\textbf{Case 2.2.1:} Suppose $\sigma'_a = \tau'_{c+1}$ and, for any $a+1 \le m_1 \le e$ and $c+2 \le m_2 \le l_2$, $\sigma'_{m_1} \neq \tau'_{m_2}$.

Then there exists an $(n-1,n)$-circuit sequence
\[
\sigma'_a, \eta_a, \sigma'_{a+1}, \dots, \sigma'_e, \eta_e, \tau'_{l_2}, \eta'_{l_2-1}, \tau'_{l_2-1}, \dots, \tau'_{c+2}, \eta'_{c+1}, \tau'_{c+1}.
\]
Therefore, by Corollary~\ref{c 3.2}, we obtain a contradiction.

\textbf{Case 2.2.2:} Suppose $\sigma'_a = \tau'_{c+1}$ and, for some $a+1 \le m_1 \le e$ and $c+2 \le m_2 \le l_2$, $\sigma'_{m_1} = \tau'_{m_2}$.

Let $d$ be the least integer $m_1$ such that $a+1 \le m_1 \le e$ and $\sigma'_{m_1} = \tau'_{m_2}$ for some $c+2 \le m_2 \le l_2$. Then there exists an $(n-1,n)$-circuit sequence
\[
\sigma'_a, \eta_a, \sigma'_{a+1}, \dots, \sigma'_{d-1}, \eta_{d-1}, \sigma'_d, \eta'_{m_2-1}, \tau'_{m_2-1}, \dots, \tau'_{c+2}, \eta'_{c+1}, \tau'_{c+1}.
\]
Therefore, by Corollary~\ref{c 3.2}, we obtain a contradiction.

\textbf{Case 2.2.3:} Suppose $\sigma'_a \neq \tau'_{c+1}$ and, for any $a \le m_1 \le e$ and $c+1 \le m_2 \le l_2$, $\sigma'_{m_1} \neq \tau'_{m_2}$.

Then there exists an $(n-1,n)$-circuit sequence
\[
\tau'_{c+1}, \eta_{a-1}, \sigma'_a, \eta_a, \sigma'_{a+1}, \dots, \sigma'_e, \eta_e, \tau'_{l_2}, \eta'_{l_2-1}, \tau'_{l_2-1}, \dots, \tau'_{c+2}, \eta'_{c+1}, \tau'_{c+1}.
\]
Therefore, by Corollary~\ref{c 3.2}, we obtain a contradiction.

\textbf{Case 2.2.4:} Suppose $\sigma'_a \neq \tau'_{c+1}$ and, for some $a\le m_1 \le e$ and $c+1 \le m_2 \le l_2$, $\sigma'_{m_1} = \tau'_{m_2}$.

Let $g$ be the least integer $m_1$ such that $a \le m_1 \le e$ and $\sigma'_{m_1} = \tau'_{m_2}$ for some $c+1 \le m_2 \le l_2$. Then there exists an $(n-1,n)$-circuit sequence
\[
\tau'_{c+1}, \eta_{a-1}, \sigma'_a, \dots, \sigma'_{g-1}, \eta_{g-1}, \sigma'_g, \eta'_{m_2-1}, \tau'_{m_2-1}, \dots, \tau'_{c+2}, \eta'_{c+1}, \tau'_{c+1}.
\]
Therefore, by Corollary~\ref{c 3.2}, we obtain a contradiction.

In all cases, we obtain a contradiction to the acyclicity of $K$. Thus, the assumption that there exists an index $h$ such that $\eta_h \neq \eta'_l$ for all $1 \le l \le s$ is false. Consequently,
\[
\{\eta_1,\dots,\eta_r\} \subseteq \{\eta'_1,\dots,\eta'_s\}.
\]
By Lemma~\ref{l 3.2}, it follows that $\eta_u = \eta'_u$ for all $1 \le u \le r$.

An analogous argument shows that there exists a unique $(n-1,n)$-path sequence between any two $(n-1)$-simplices in $K$.

Conversely, let $K$ be a connected pure $n$-simplicial complex satisfying conditions~(\ref{un1}), (\ref{un2}), and~(\ref{un3}). Assume there exists an $(n-1,n)$-Dewdney cycle sequence
\[
\sigma_1, \eta_1, \sigma_2, \dots, \sigma_r, \eta_r, \sigma_1.
\]
Then $\sigma_1, \eta_1, \sigma_2$ and $\sigma_1, \eta_r, \sigma_r, \dots, \sigma_3, \eta_2, \sigma_2$ are distinct $(n-1,n)$-path sequences, yielding a contradiction. Therefore, $K$ is acyclic and hence an $(n-1,n)$-tree.

\end{proof}

From the proof of the converse part of Theorem~\ref{thmun}, we obtain the following corollary.

\begin{cor}
Let $K$ be a connected pure $n$-simplicial complex satisfying the following properties:
\begin{enumerate}
    \item \label{cc1} There exists a unique $(n-1)$-path sequence between any two $(n-1)$-simplices.
    \item \label{cc2} For all $0 \le m \le n-2$, there exists no $(m,n)$-Dewdney cycle sequence.
\end{enumerate}
Then $K$ is an $(n-1,n)$-tree.
\end{cor}

The following example shows that condition (\ref{un3}) in Theorem~\ref{thmun} is necessary and cannot be omitted. In Figure~\ref{fig9}, the simplicial complex 
\[
K = \overline{\{\{1, 2, 3\}, \{2, 3, 4\}, \{3, 4, 5\}, \{1, 4, 5\}, \{1, 2, 6\}, \{1, 3, 7\}, \{1, 4, 9\}, \{1, 5, 8\}\}}
\]
is connected and satisfies conditions (\ref{un1}) and (\ref{un2}) of Theorem~\ref{thmun}, which can be verified directly from the structure of $K$. However, the $(0,2)$-walk sequence
\[
\{1\}, \{1,2,3\}, \{2\}, \{2,3,4\}, \{3\}, \{3,4,5\}, \{5\}, \{1,4,5\}, \{1\}
\]
forms a $(0,2)$-Dewdney cycle sequence in \( K \). Consequently, \( K \) is not an $(1,2)$-tree. Therefore, if condition (\ref{un3}) of Theorem~\ref{thmun} is dropped, then the converse of Theorem~\ref{thmun} fails.

\begin{figure}[h] 
    \centering
    \includegraphics[width=0.8\textwidth]{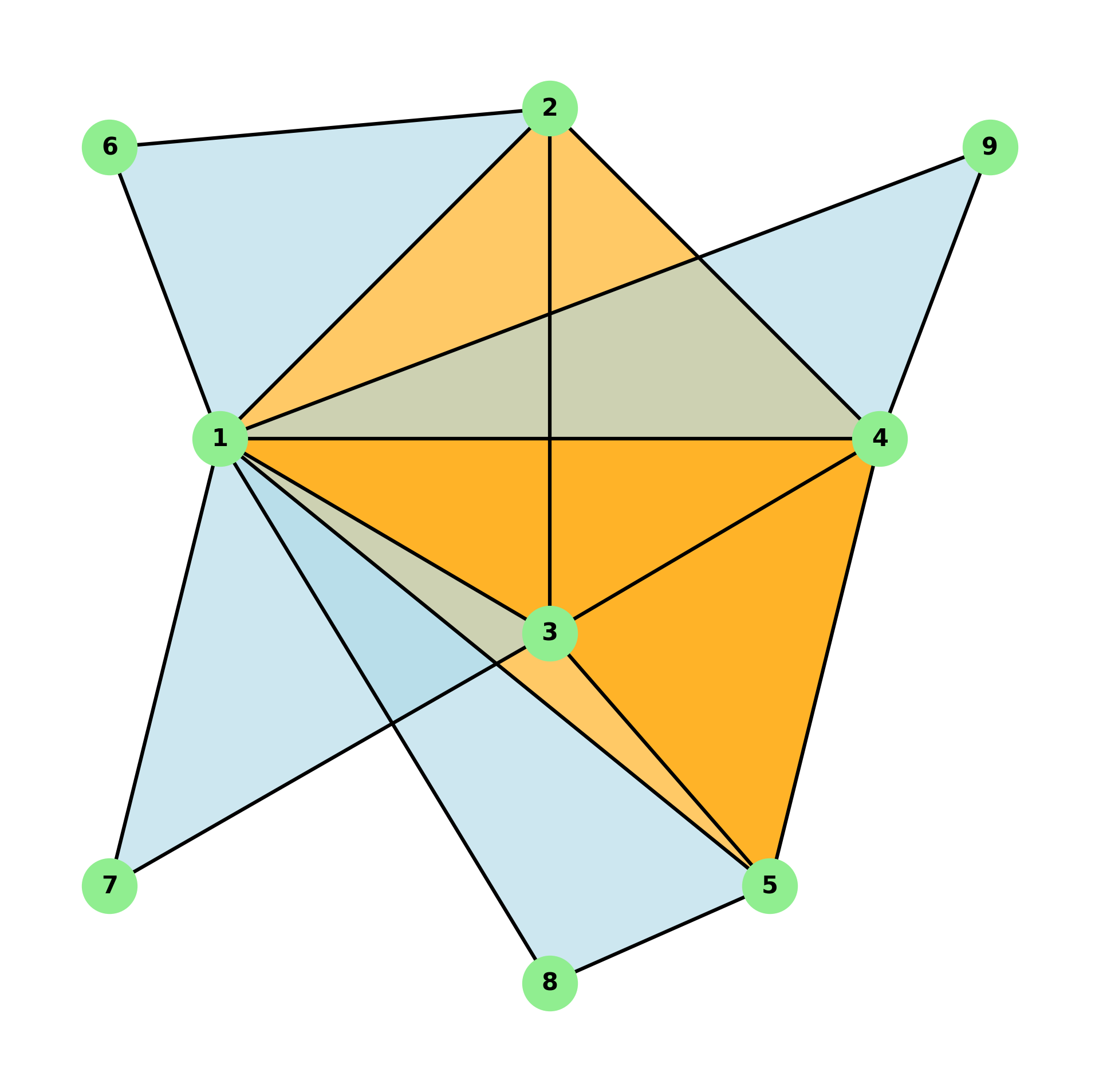} 
    \caption{A connected pure $2$-simplicial complex satisfying conditions~(\ref{un1}) and~(\ref{un2}) of Theorem~\ref{thmun}, but containing a $(0,2)$-Dewdney cycle sequence. This example shows that condition~(\ref{un3}) is necessary for the converse direction.}
    \label{fig9} 
\end{figure}

We now turn to enumerative characterisations, completing the analogy with graph-theoretic trees. The following results, namely Theorem~\ref{1 tree simplices information}, Lemma~\ref{l 3.3}, and Lemma~\ref{l 3.4}, which are central to our work, are directly adapted from Dewdney's study of $(m,n)$-trees~\cite{dewdney1974higher} in the case $m=n-1$. 

\begin{theorem}\label{1 tree simplices information}
    If $K$ is an $(n-1,n)$-tree with $p$ vertices, then for all $1 \leq k \leq n$, $$\alpha_k(K) = (p - n) \binom{n}{k} + \binom{n}{k+1}.$$

\end{theorem}

\noindent \textbf{Remark:} The formula in Theorem~\ref{1 tree simplices information} also holds for $k=0$, yielding $\alpha_0(K)=p$.

\begin{lemma}
\label{l 3.3}
    If $K$ is an $(n-1,n)$-tree and $M$ is a $k$-complete subcomplex with $p$ vertices, where $1 \leq k \leq n \text{ and }  k+1 \leq p\leq n + 1$, then there exists an $n$-simplex $\eta_q$ in $K$ such that $M$ is a subcomplex of $\overline{\{\eta_q\}}$.
\end{lemma}

\begin{lemma}
\label{l 3.4}
    If $K$ is a connected pure $n$-simplicial complex with $p$ vertices, then  $$\alpha_k(K) \geq (p-n)\binom{n}{k}+\binom{n}{k+1} \text{ for all } 1 \leq k \leq n-1.$$ The equality holds only if $K$ is an $(n-1,n)$-tree. 
\end{lemma}

The following Lemma~\ref{l 3.5} is introduced to refine Dewdney’s counting characterisations and to address the limitations identified in Theorem~\ref{1.1}.

\begin{lemma}
\label{l 3.5}
Let $K$ be a pure $n$-simplicial complex. Then:
\begin{enumerate}
    \item If $K$ has no $(n-1,n)$-circuit sequence, then $\alpha_{n-1}(K) \ge n \alpha_n(K) + 1$.
    \item If $K$ is connected, then $\alpha_{n-1}(K) \le n \alpha_n(K) + 1$.
\end{enumerate}
\end{lemma}

\begin{proof}
Let $K$ be a pure $n$-simplicial complex with no $(n-1,n)$-circuit sequence and $\alpha_n(K)=t$. We prove $\alpha_{n-1}(K) \ge n \alpha_n(K) + 1$ by induction on $t$.

If $t = 1$ or $t = 2$, the result holds trivially. Assume $t \ge 3$, and that the result holds for any pure $n$-simplicial complex with no $(n-1,n)$-circuit sequence and $t-1$ $n$-simplices.

We first show that $K$ contains an $n$-simplex with at least $n$ $(n-1)$-faces, none of which is a face of any other $n$-simplex.

Suppose not. Then every $n$-simplex in $K$ shares at least two $(n-1)$-faces with other $n$-simplices. Begin with an $(n-1)$-face $\sigma_1$ of an $n$-simplex $\eta_1$. Next, choose another $(n-1)$-face $\sigma_2$ of $\eta_1$ that is also a face of another $n$-simplex $\eta_2$. Continue by selecting another $(n-1)$-face $\sigma_3$ of $\eta_2$ that is a face of yet another $n$-simplex $\eta_3$. This process forms a finite-length $(n-1,n)$-walk sequence
\[
\sigma_1, \eta_1, \sigma_2, \dots, \sigma_r, \eta_r, \sigma_{r+1},
\]
in which at least one $n$-simplex or $(n-1)$-simplex occurs more than once. We observe that any three consecutive $(n-1)$-simplices and $n$-simplices in the sequence are distinct. This follows from the construction, since at each step the $(n-1)$-simplex $\sigma_{s+1}$ is chosen to be distinct from $\sigma_s$, and the $n$-simplex $\eta_{s+1}$ is chosen to be distinct from $\eta_s$.

Let $k$ be the smallest integer $j$ such that $4 \le j \le r+1$ and $\sigma_j = \sigma_i$ for some $1 \le i \le j-3$. Similarly, let $l$ be the smallest integer $q$ such that $4 \le q \le r$ and $\eta_q = \eta_p$ for some $1 \le p \le q-3$.

If $k \le l$ or if $l$ does not exist, then the sequence
\[
\sigma_i, \eta_i, \sigma_{i+1}, \dots, \sigma_{k-1}, \eta_{k-1}, \sigma_k
\]
forms an $(n-1,n)$-circuit sequence, yielding a contradiction. If $k > l$ or if $k$ does not exist, then the sequence
\[
\sigma_l, \eta_p, \sigma_{p+1}, \dots, \sigma_{l-1}, \eta_{l-1}, \sigma_l
\]
forms an $(n-1,n)$-circuit sequence, again yielding a contradiction. Therefore, $K$ contains an $n$-simplex $\eta_s$ with at least $n$ $(n-1)$-faces that are not faces of any other $n$-simplex.

Let $L = \overline{\{\eta_1, \dots, \eta_t\} \setminus \{\eta_s\}}$. Then $L$ is a pure $n$-simplicial complex with $t-1$ $n$-simplices and no $(n-1,n)$-circuit sequence. By the induction hypothesis,
\[
\alpha_{n-1}(L) \ge n \alpha_n(L) + 1.
\]
Since $\eta_s$ has at least $n$ $(n-1)$-faces that are not faces of any other $n$-simplex, removing $\eta_s$ decreases the number of $(n-1)$-simplices by at least $n$. Therefore,
\[
\alpha_{n-1}(K) \ge n \alpha_n(L) + 1 + n = n \alpha_n(K) + 1.
\]

Now we prove the second part of the lemma. Let $\eta_1, \dots, \eta_t$ be an $(n-1)$-ordering of $K$. Then $\eta_1$ has $n+1$ $(n-1)$-faces, and for all $2 \le i \le t$, $\eta_i$ has at most $n$ $(n-1)$-faces that are not faces of $\eta_1, \dots, \eta_{i-1}$. Therefore,
\[
\alpha_{n-1}(K) \le n \alpha_n(K) + 1.
\]
\end{proof}

We now improve Dewdney’s enumerative characterisation given in Theorem~\ref{1.1} by showing that, in the case $m = n-1$, the required counting condition can be verified even at the top-dimensional level $k = n$.

\begin{theorem}
\label{3.7}
A pure $n$-simplicial complex $K$ with $p$ vertices is an $(n-1,n)$-tree if and only if it satisfies the following two properties:
\begin{enumerate}
    \item $K$ is connected.
    \item $\alpha_k(K) = (p - n) \binom{n}{k} + \binom{n}{k+1}$ for some $1 \le k \le n$.
\end{enumerate}
\end{theorem}

\begin{proof}
Let $K$ be a connected pure $n$-simplicial complex such that
\[
\alpha_k(K) = (p - n) \binom{n}{k} + \binom{n}{k+1}
\]
for some $1 \le k \le n$.

If $1 \le k \le n-1$, then by Lemma~\ref{l 3.4}, $K$ is an $(n-1,n)$-tree. Now consider the case $k = n$. By Lemma~\ref{l 3.5}, we have
\[
\alpha_{n-1}(K) \le n(p - n) + 1.
\]
Furthermore, by Lemma~\ref{l 3.4}, we obtain
\[
\alpha_{n-1}(K) \ge (p - n)n + 1.
\]
Combining these inequalities, we conclude that
\[
\alpha_{n-1}(K) = (p - n)n + 1.
\]
Therefore, by Lemma~\ref{l 3.4}, $K$ is an $(n-1,n)$-tree.
\end{proof}

We next present an example showing that the conditions proposed in Conjectures~\ref{C1} and~\ref{C2} are not sufficient to characterise $(m,n)$-trees. Since both conjectures are formulated for all $1 \le m \le n-1$, it suffices to provide a counterexample for a particular choice of $(m,n)$. In the case $m=1$, the conditions in Conjectures~\ref{C1} and~\ref{C2} coincide, as the index $k$ can only take the value $1$. Hence, the following example disproves the ``if'' direction of both conjectures.

Consider the pure $2$-simplicial complex
\[
K = \overline{\{\{1,2,3\}, \{2,5,6\}, \{3,4,5\}\}},
\]
where the vertex set is $V(K)=\{1,2,3,4,5,6\}$ (see Figure~\ref{fig3}). Then $\alpha_0(K)=6$ and $\alpha_1(K)=9$. The complex has no $(1,2)$-circuit sequences and satisfies the enumerative identity in Conjectures~\ref{C1} and~\ref{C2}. However, $K$ is disconnected and hence is not an $(1,2)$-tree. This provides a counterexample to the ``if'' direction of both conjectures.


\begin{figure}[h] 
    \centering
    \includegraphics[width=0.5\textwidth]{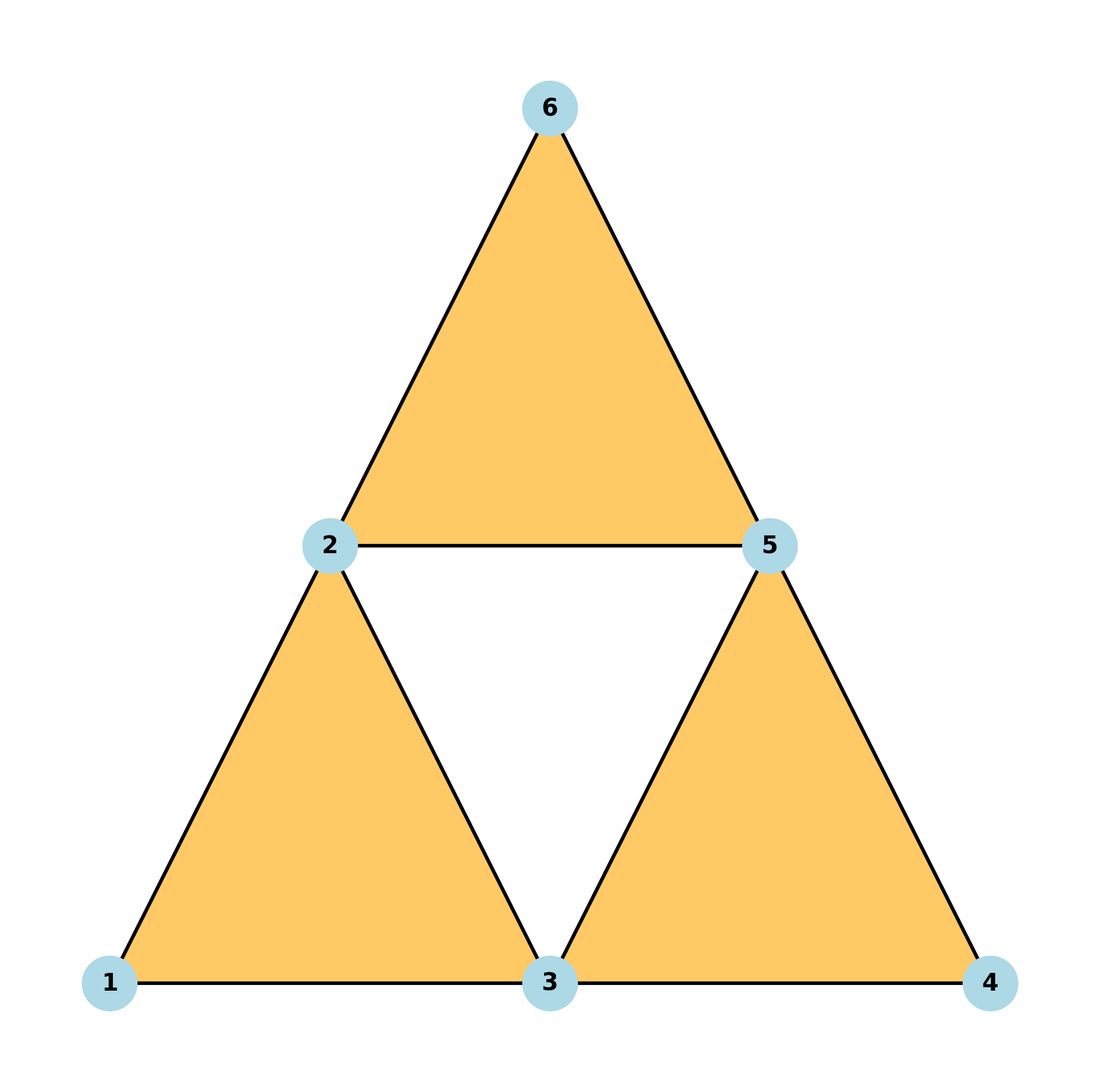} 
    \caption{A pure $2$-simplicial complex $K = \overline{\{\{1,2,3\}, \{2,5,6\}, \{3,4,5\}\}}$ illustrating a counterexample to Conjectures~\ref{C1} and~\ref{C2}.}
    \label{fig3} 
\end{figure}

Before stating a modified version of Dewdney's conjecture, we establish a simple relationship between circuit sequences of different dimensions. The following lemma shows that the existence of an $(n-1,n)$-circuit sequence forces the existence of an $(m,n)$-circuit sequence for every $0 \leq m \leq n-2$.

\begin{lemma}\label{1rt}
   If a pure $n$-simplicial complex $K$ contains an $(n-1,n)$-circuit sequence, then for every $0 \leq m \leq n-2$, it also contains an $(m,n)$-circuit sequence.
\end{lemma}

\begin{proof}
Let
$$
\sigma_i=\sigma_1,\eta_1,\sigma_2,\eta_2,\ldots,\sigma_r,\eta_r,\sigma_{r+1}=\sigma_i
$$
be an $(n-1,n)$-circuit sequence, and let $0\le m\le n-2$.

Choose an arbitrary $m$-face $\tau_1$ of $\sigma_1$. Since $\sigma_2$ is distinct from $\sigma_1$, there exists an $m$-face $\tau_2$ of $\sigma_2$ such that $\tau_2\neq\tau_1$. If $\tau_1$ is a face of $\eta_2$, then
$$
\tau_1,\eta_1,\tau_2,\eta_2,\tau_1
$$
is an $(m,n)$-circuit sequence, and we are done.

Assume therefore that $\tau_1$ is not a face of $\eta_2$. Choose an $m$-face $\tau_3$ of $\sigma_3$ with $\tau_3\neq\tau_2$. If $\tau_1$ is a face of $\eta_3$, then
$$
\tau_1,\eta_1,\tau_2,\eta_2,\tau_3,\eta_3,\tau_1
$$
is an $(m,n)$-circuit sequence.

Suppose now that $\tau_1$ is not a face of $\eta_3$. Choose an $m$-face $\tau_4$ of $\sigma_4$ with $\tau_4\neq\tau_3$. If $\tau_4=\tau_2$, then
$$
\tau_2,\eta_2,\tau_3,\eta_3,\tau_4
$$
is an $(m,n)$-circuit sequence. Otherwise, if $\tau_4\neq\tau_2$ and $\tau_1$ is a face of $\eta_4$, then
$$
\tau_1,\eta_1,\tau_2,\eta_2,\tau_3,\eta_3,\tau_4,\eta_4,\tau_1
$$
is an $(m,n)$-circuit sequence.

If neither of these possibilities occurs, we continue inductively. At the $k$-th step, choose an $m$-face $\tau_k$ of $\sigma_k$ distinct from $\tau_{k-1}$. If $\tau_k$ coincides with one of the previously chosen faces, then the corresponding subsequence determines an $(m,n)$-circuit sequence. Otherwise, if $\tau_1$ is a face of $\eta_k$, then
$$
\tau_1,\eta_1,\tau_2,\eta_2,\ldots,\tau_k,\eta_k,\tau_1
$$
is an $(m,n)$-circuit sequence.

Since the given $(n-1,n)$-circuit sequence has finite length, this process cannot continue indefinitely. Hence, after finitely many steps, one of the above two situations must occur, yielding an $(m,n)$-circuit sequence. Therefore, an $(m,n)$-circuit sequence exists.
\end{proof}

We now present a modified version of Dewdney’s conjecture for the case $m = n-1$. The modification consists of strengthening the circuit-freeness assumption and replacing the original enumerative identity by explicit counting conditions on $(n-1)$- and $n$-simplices. The resulting formulation restores the classical correspondence between acyclicity, counting, and tree structure in the $(n-1,n)$-setting. 

\begin{theorem}\label{thm5.5}
A pure $n$-simplicial complex $K$ with $p$ vertices is an $(n-1,n)$-tree if and only if it satisfies the following three properties:
\begin{enumerate}
    \item\label{1403} $K$ has no $(m,n)$-circuit sequences for some $0 \le m \le n-1$.
    \item\label{222} $\alpha_{n-1}(K) = (p - n)n + 1$.
    \item\label{9916} $\alpha_n(K) = p - n$.
\end{enumerate}
\end{theorem}

\begin{proof}
Let $K$ be a pure $n$-simplicial complex satisfying properties~(\ref{1403}), (\ref{222}), and~(\ref{9916}). It follows from Lemma~\ref{1rt} that $K$ contains no $(n-1,n)$-circuit sequences.

Let $L_1, \dots, L_r$ denote the components of $K$ for some $r \ge 1$, so that $K = \bigcup_{i=1}^r L_i$. For all $1 \le i \le r$, since $L_i$ is connected and contains no $(n-1,n)$-circuit sequences, Lemma~\ref{l 3.5} implies
\[
\alpha_{n-1}(L_i) = n \alpha_n(L_i) + 1. \tag{i} \label{eq:1}
\]

Summing~\eqref{eq:1} over $i = 1, \dots, r$, we obtain
\[
\alpha_{n-1}(K)
= \sum_{i=1}^r \bigl( n \alpha_n(L_i) + 1 \bigr)
= n \sum_{i=1}^r \alpha_n(L_i) + r
= n(p - n) + r. \tag{ii} \label{eq:11}
\]

Comparing~\eqref{eq:11} with the identity $\alpha_{n-1}(K) = (p - n)n + 1$, we conclude that $r = 1$. Hence, $K$ is connected. By Theorem~\ref{3.7}, $K$ is an $(n-1,n)$-tree.
\end{proof}



\begin{figure}
    \centering
    \includegraphics[width=0.5\linewidth]{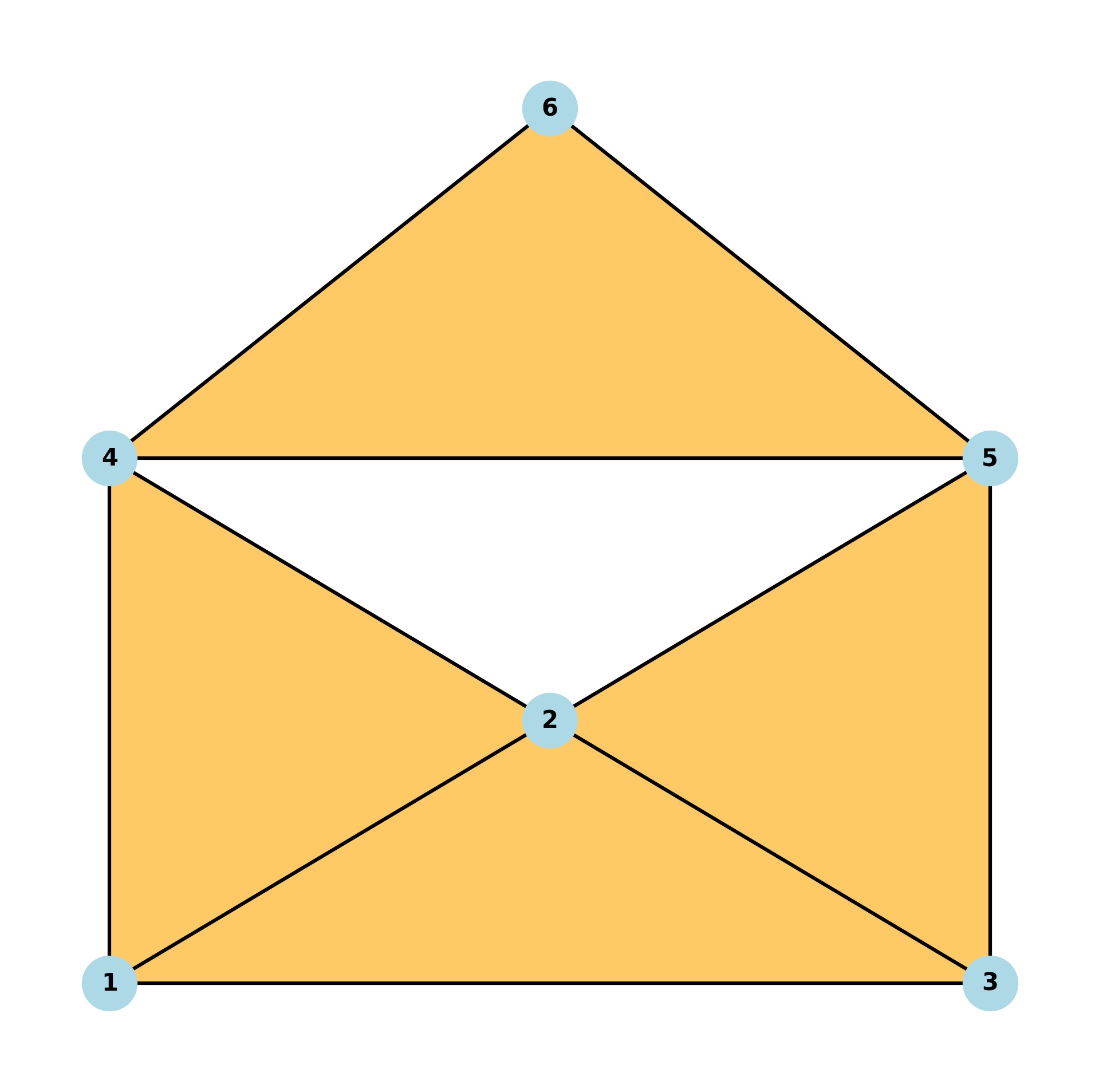}
    \caption{Disconnected pure $2$-simplicial complex $K=\overline{\{\{1,2,4\},\{1,2,3\},\{2,3,5\},\{4,5,6\}\}}$ with six vertices, four 2-simplices, and contains no $(1,2)$-circuit sequence.}
    \label{fig:4}
\end{figure}

To see that condition (\ref{222}) in Theorem \ref{thm5.5} is necessary and cannot be omitted, consider the example in Figure~\ref{fig:4}. In this figure, a pure $2$-simplicial complex $$K=\overline{\{\{1,2,4\},\{1,2,3\},\{2,3,5\},\{4,5,6\}\}}$$is disconnected, consisting of six vertices and four $2$-simplices, with no $(1,2)$-circuit sequence. This shows that without condition (\ref{222}), the complex satisfies conditions (\ref{1403}) and (\ref{9916}) of Theorem \ref{thm5.5} but fails to be an $(1,2)$-tree due to being disconnected.

However, a natural question arises: can condition~(\ref{222}) in Theorem~\ref{thm5.5} be relaxed by replacing it with the knowledge of $\alpha_k(K)$ for some $1 \leq k \leq n-1$?

For $n=2$, the above question is true, as it follows from Theorem~\ref{thm5.5}, where knowledge of $\alpha_1(K)$ suffices. This naturally leads to the following question of whether the same holds for $n \geq 3$.

\begin{question}
Let $K$ be a pure $n$-simplicial complex with $p$ vertices. Suppose 
that for some $1 \leq k \leq n-1$, 
\[
\alpha_k(K) = (p - n)\binom{n}{k} + \binom{n}{k+1}
\quad \text{and} \quad \alpha_n(K) = p - n,
\]
and that $K$ has no $(m,n)$-circuit sequences for some 
$0 \leq m \leq n-1$. 

Does it follow that $K$ is an $(n-1,n)$-tree?
\end{question}

Our current framework does not provide a proof for this, nor have we been able to construct a counterexample, and we leave it as an open question for future work.

\section*{Data availability}
This work is purely theoretical and does not use or generate any datasets. Hence, no data are available for sharing.

\section*{Conflict of interest}
The author declares no competing or conflicting interests.

\appendix
\section*{Appendix}

\subsection*{Examples distinguishing different notions of higher-dimensional trees}

In this appendix, we present examples illustrating that different notions of higher-dimensional trees appearing in the literature are not equivalent.

We first recall the notion of a hypertree introduced by Kalai~\cite{kalai1983enumeration, linial2019enumeration}.

\begin{definition*}
A $d$-hypertree $T$ is a $d$-dimensional simplicial complex with a full $(d-1)$-skeleton such that
\[
H_{d-1}(T;\mathbb{Q})=0
\quad \text{and} \quad
H_d(T;\mathbb{Q})=0.
\] 
\end{definition*} 

Here, the conditions $H_{d-1}(T;\mathbb{Q})=0$ and $H_d(T;\mathbb{Q})=0$ mean that $T$ has trivial homology in dimensions $d-1$ and $d$ over $\mathbb{Q}$. Moreover, a simplicial complex is said to have a full $(d-1)$-skeleton if it contains every $(d-1)$-simplex determined by its vertex set.

We now show that there exist $(m,n)$-trees that are not hypertrees in the sense of Kalai.

Consider the simplicial complex shown in Figure~\ref{1fig:comparison_trees}(a):
\[
H_1=\{\{1\},\{2\},\{3\},\{4\},\{1,2\},\{1,3\},\{2,3\},\{2,4\},\{3,4\},\{1,2,3\},\{2,3,4\}\}.
\]
By Definition~\ref{1 d mn tree}, $H_1$ is a $(1,2)$-tree.

However, $H_1$ is not a $2$-hypertree in the sense of Kalai, since it does not have a full $1$-skeleton. In particular, the edge $\{1,4\}$ is not a simplex of $H_1$.

\begin{figure}[h]
    \centering
    \begin{subfigure}{0.45\textwidth}
        \centering
        \includegraphics[width=\linewidth]{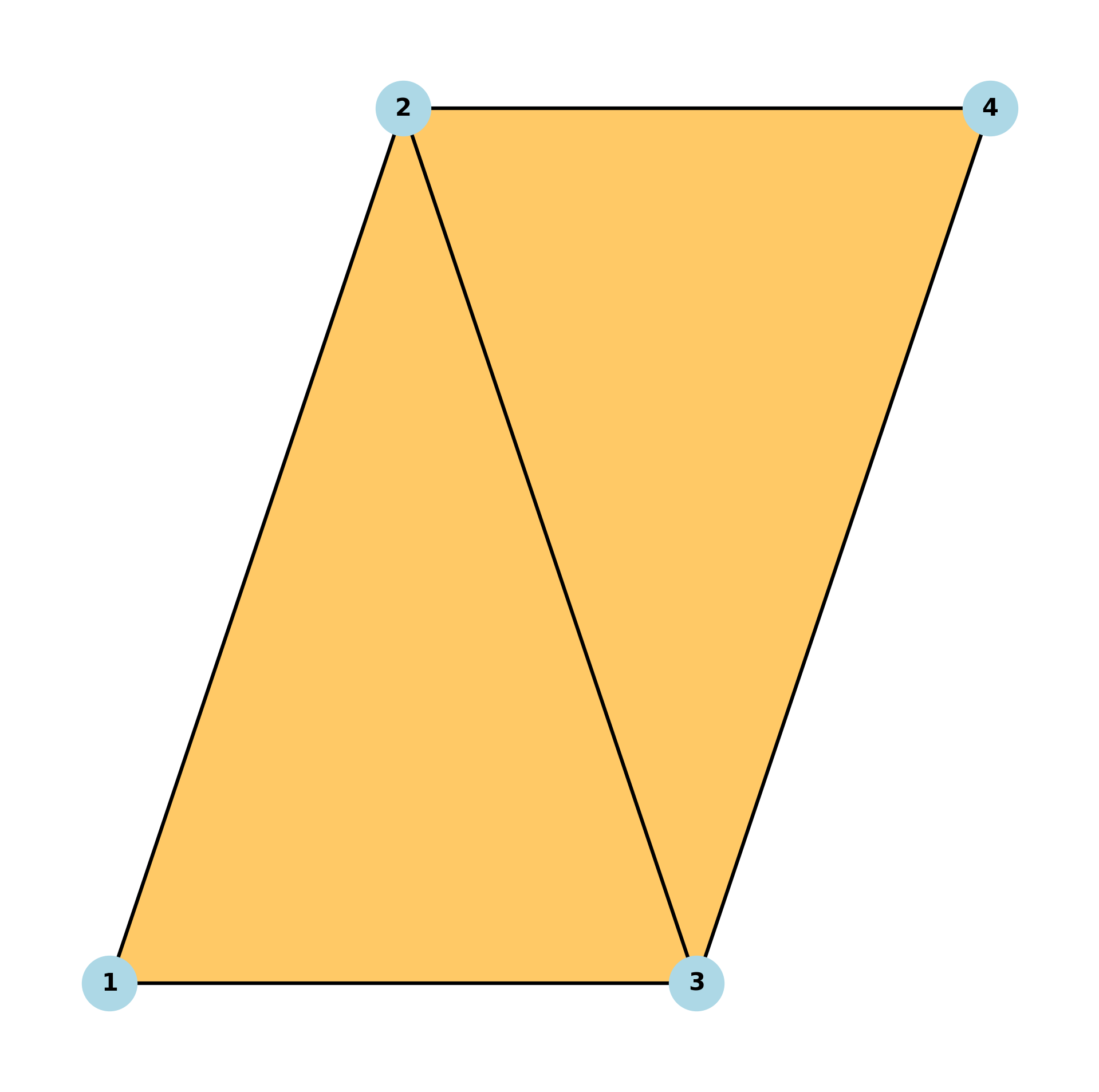}
        \caption{}
    \end{subfigure}
    \hfill
    \begin{subfigure}{0.45\textwidth}
        \centering
        \includegraphics[width=\linewidth]{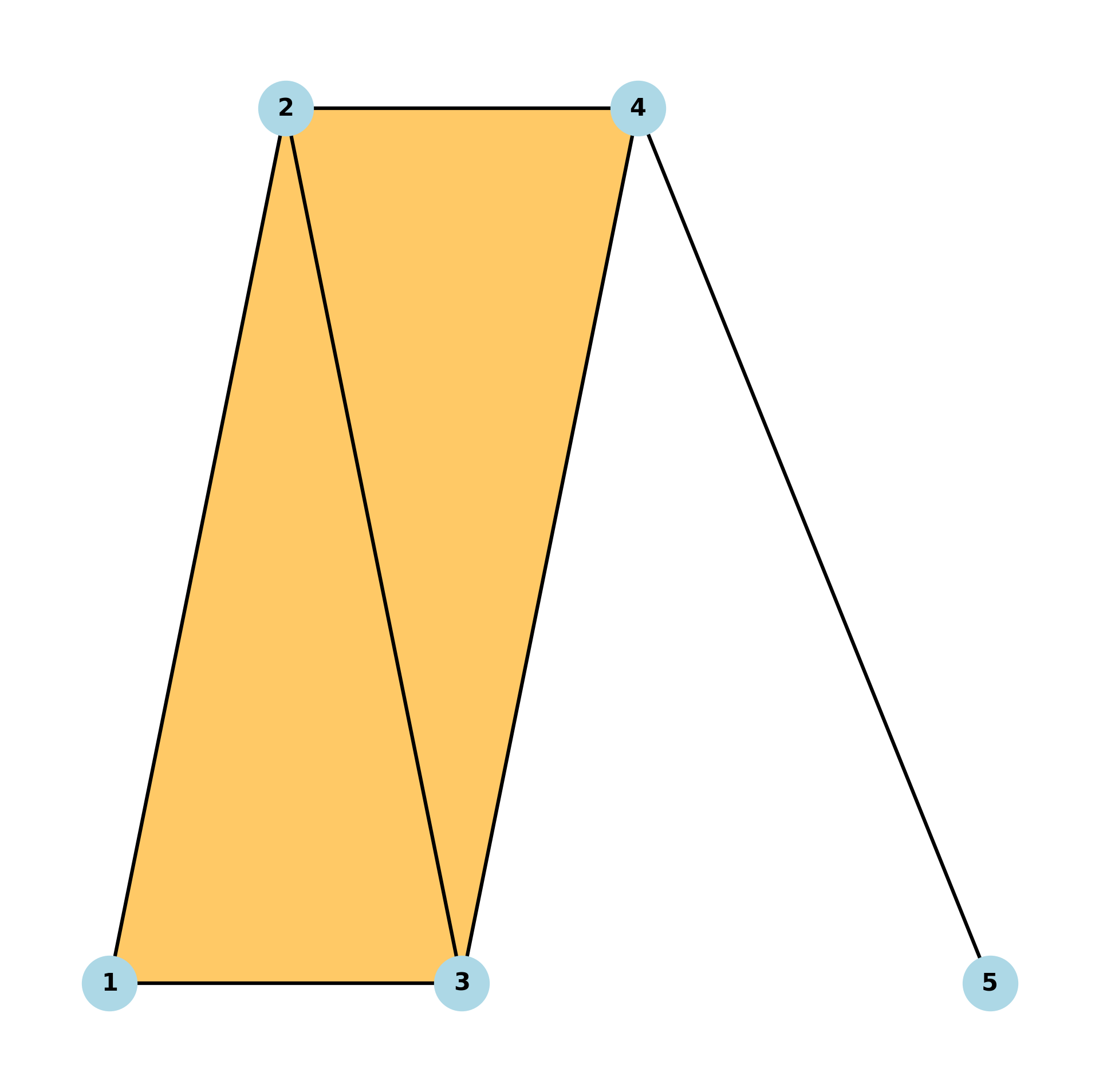}
        \caption{}
    \end{subfigure}
    
    \caption{(a) A $(1,2)$-tree which is not a $2$-hypertree in the sense of Kalai. (b) A simplicial tree in the sense of Faridi which is not an $(m,n)$-tree in the sense of Dewdney.}
    \label{1fig:comparison_trees}
\end{figure}

We next show that simplicial trees in the sense of Faridi need not be $(m,n)$-trees in the sense of Dewdney.

For the definition of simplicial trees introduced by Faridi, we refer the reader to \cite{caboara2007simplicial}, since including the definition here would require introducing additional notions that are not used elsewhere in the present paper.

Consider the simplicial complex shown in Figure~\ref{1fig:comparison_trees}(b):
\[
H_2=\{\{1\},\{2\},\{3\},\{4\},\{5\},\{1,2\},\{1,3\},\{2,3\},\{2,4\},\{3,4\},\{4,5\},\{1,2,3\},\{2,3,4\}\}.
\]
This example is taken from \cite{caboara2007simplicial}, where it is shown that $H_2$ is a simplicial tree in the sense of Faridi.

However, $H_2$ is not an $(m,n)$-tree in the sense of Dewdney for any $m$ and $n$, since Dewdney's $(m,n)$-trees are defined only for pure simplicial complexes, whereas $H_2$ is not pure.



\end{document}